\documentclass[preprint,12pt]{elsarticle}

\usepackage{amsmath,amssymb,amsthm}

\usepackage{enumerate}
\usepackage{mathtools}
\usepackage{graphicx}
\usepackage{float}

\newtheorem{theorem}{Theorem}[section]
\newtheorem{thm}{Theorem}[section]

\newtheorem{lemma}[theorem]{Lemma}
\newtheorem{definition}[theorem]{Definition}
\newtheorem{dfn}[theorem]{Definition}

\newtheorem{alg}[theorem]{Algorithm}
\newtheorem{remark}[theorem]{Remark}
\newtheorem{assumption}[theorem]{Assumption}

\newcommand{\nr}[1]{\ensuremath{\left\|{#1}\right\|}}
\newcommand{\nrs}[1]{\ensuremath{\|{#1}\|}} 

\journal{arXiv}

\begin{document}

\begin{frontmatter}



\title{ Enabling fast convergence of the iterated penalty Picard iteration with $O(1)$ penalty parameter for incompressible Navier-Stokes via Anderson acceleration}


\author[label1]{Leo G. Rebholz \corref{cor1}\fnref{lab1}}
\ead{rebholz@clemson.edu}
\author[label1]{Duygu Vargun \fnref{lab1}}
\ead{dvargun@clemson.edu}
\address[label1]{Department of Mathematical Sciences, Clemson University, Clemson, SC 29634, USA}
\fntext[lab1]{This  author  was  partially  supported  by  NSF  Grant  DMS 2011490.}

\author[label2]{Mengying Xiao}
\ead{mxiao@uwf.edu}
\address[label2]{Department of Mathematics and Statistics, University of West Florida, Pensacola, FL 32514, USA}

\cortext[cor1]{Corresponding  author.}

\begin{abstract}
This paper considers an enhancement of the classical iterated penalty Picard (IPP) method for the incompressible Navier-Stokes equations, where we restrict our attention to $O(1)$ penalty parameter, and Anderson acceleration (AA) is used to significantly improve its convergence properties. After showing the fixed point operator associated with the IPP iteration is Lipschitz continuous and Lipschitz continuously (Frechet) differentiable, we apply a recently developed general theory for AA to conclude that IPP enhanced with AA improves its linear convergence rate by the gain factor associated with the underlying AA optimization problem. Results for several challenging numerical tests are given and show that IPP with penalty parameter 1 and enhanced with AA is a very effective solver. 
\end{abstract}

%

%
%
%

\end{frontmatter}


\section{Introduction}
We consider solvers for the incompressible Navier-Stokes equations (NSE), which are given by

\begin{align}
u_t + u\cdot\nabla u + \nabla p - \nu\Delta u & = f, \label{ns1} \\
\nabla \cdot u & = 0, \label{ns2} 
\end{align}

where $u$ and $p$ are the unknown velocity and pressure, $\nu$ is the kinematic viscosity which is inversely proportional to the Reynolds number $Re$, and $f$ is a known function representing external forcing. For simplicity we assume no-slip boundary conditions and a steady flow ($u_t =0$) as well as small data so as to be consistent with steady flow, but our analysis and results can be extended to other common boundary conditions and temporarily discretized transient flows with only minor modifications.  Due to the wide applicability of \eqref{ns1}-\eqref{ns2} across science and engineering, many nonlinear solvers already exist for it \cite{JKN17}, with the most popular being Picard and Newton iterations \cite{GR86}. Newton's iteration converges quadratically once near a root, but requires a good initial guess, especially for higher $Re$ \cite{GR86}. The Picard iteration for the NSE is linearly convergent, but also globally convergent and is much more robust for higher $Re$ \cite{GR86, PR21}. 

Herein, we consider Anderson acceleration (AA) of the iterated penalty Picard (IPP) iteration. The IPP iteration is generally more efficient than Picard for a single iteration since the linear solve is easier/cheaper, but compared to Picard it can be less robust and require more iterations if the penalty parameter is not chosen correctly.    The IPP scheme for the NSE is given in \cite{C93} as: Given $u_k, p_k$, solve for $u_{k+1},p_{k+1}$ from 
\begin{align}
u_{k}\cdot\nabla u_{k+1} + \nabla p_{k+1} - \nu\Delta u_{k+1} & = f, \label{ns4} \\
\epsilon p_{k+1} + \nabla \cdot u_{k+1} & = \epsilon p_k, \label{ns5} 
\end{align}
where $\epsilon>0$ is a penalty parameter, generally taken small.    The system \eqref{ns4}-\eqref{ns5} is equivalent to the velocity-only system
\begin{equation}
u_{k}\cdot\nabla u_{k+1} - \epsilon^{-1}  \nabla (\nabla \cdot u_{k+1})  - \nu\Delta u_{k+1}  = f + \epsilon^{-1}  \sum_{j=0}^{k} \nabla (\nabla \cdot u_j), \label{ns7}
\end{equation}
which is used in \cite{G89It, S21, HS18, RX15}, and the pressure can be expressed in terms of velocities, i.e. $p_{k+1} = -\epsilon^{-1} \sum\limits_{j=0}^{k+1} \nabla \cdot u_j$.
There are several advantages to using the IPP, including the pressure in the continuity equation allows for circumventing the inf-sup condition on the velocity and pressure spaces \cite{C93}, and Scott-Vogelius elements can be used without any mesh restriction and will produce a pointwise divergence free solution along with many advantages this brings.
Codina showed in \cite{C93} that a discretization of \eqref{ns4}-\eqref{ns5} converges linearly under a small data condition and sufficiently small $\epsilon$, and has a better convergence rate if the penalty parameter $\epsilon$ is chosen sufficiently small.  

Unfortunatly, with small $\epsilon$ the advantages of using \eqref{ns4}-\eqref{ns5} diminish since the same nonsymmetric saddle point system of the usual Picard iteration is recovered as $\epsilon\rightarrow 0$, and if IPP is computed via \eqref{ns7}, then $\epsilon<1$ can lead to linear systems that most common preconditioned iterative linear solvers will have difficulty resolving \cite{S99,OT14}.  Hence even though the IPP is theoretically effective when $\epsilon$ is small, its use has largely died out over the past few decades since small $\epsilon$ leads to the need for direct linear solvers, but direct linear solvers are not effective on most large scale problems of modern interest.  Hence, in an effort to show (properly enhanced with AA) IPP can still be a very competitive solver on any size problem, we completely avoid the notion of small $\epsilon$ and in our numerical tests use only $\epsilon=1$, where preconditioned iterative methods found success on linear systems resembling \eqref{ns7} \cite{HR13,OT14,CLLRW13,BL12}.

This paper presents an analytical and numerical study of AA applied to IPP, without assuming small $\epsilon$.  AA has recently been used to improve convergence and robustness of solvers for a wide range of problems including various types of flow problems \cite{LWWY12,PRX19,PRX21}, geometry optimization \cite{PDZGQL18}, radiation diffusion and nuclear physics \cite{AJW17,TKSHCP15}, machine learning \cite{GS19}, molecular interaction \cite{SM11}, computing nearest correlation matrices \cite{HS16}, and many others e.g. \cite{WaNi11,K18,LW16,LWWY12,FZB20,WSB21}. In particular, AA was used in \cite{PRX19} to make the Picard iteration for \eqref{ns1}-\eqref{ns2} more robust with respect to $Re$ and to converge significantly faster.  Hence it is a natural and important next step to consider AA applied to IPP, which is a classical NSE solver but is not always effective when $\epsilon < 1$ due to linear solver difficulties.  Herein we formulate IPP equipped with a finite element discretization as a fixed point iteration $u_{k+1}=G(u_k)$, where $G$ is a solution operator to discrete linear system.  We then prove that $G$ is continuously (Frechet) differentiable, allowing us to invoke the AA theory from \cite{PR21}, which implies AA will improve the linear convergence rate of the iteration by a factor (less than 1) representing the gain of the underlying AA optimization problem.  Results of several numerical tests are also presented, which shows IPP using $\epsilon=1$ and enhanced with AA can be a very effective solver for the NSE.

This paper is arranged as follows: In Section 2, we provide notation and mathematical preliminaries on the finite element discretizations and AA. In section 3, we present the IPP method and prove associated fixed point solution operator properties. In section 4, we give the Anderson accelerated IPP scheme and present a convergence result. In section 5, we report on the results of several numerical tests, which demonstrate a significant (and sometimes dramatic) positive impact on the convergence.

\section{Notation and preliminaries}
We consider a domain $\Omega \subset \mathbb{R}^d$ ($d = 2,3$) that is open, connected, and with Lipschitz boundary $\partial \Omega$. The $L^2(\Omega)$ norm and inner product will be denoted by $\| \cdot \|$ and $(\cdot ,\cdot)$. Throughout this paper, it is understood by context whether a particular space is scalar or vector valued, and so we do not distinguish notation.

The natural function spaces for velocity and pressure in this setting are given by
\begin{align*}
X:=& H_0^{1}(\Omega) = \{ v\in L^2(\Omega)\mid \nabla v \in L^2(\Omega), v|_{\partial\Omega} = 0\},\\
Q:=& L_0^2(\Omega) = \{ q \in L^2(\Omega) \mid \int_\Omega q \ dx = 0\}.
\end{align*}
In the space $X$,  the Poincar\'e inequality holds \cite{laytonbook}: there exists a constant $C_P>0$ depending only on $\Omega$ such that for any $\phi\in X$,
\[
\| \phi \| \le C_P \| \nabla \phi \|.
\]
The dual space of $X$ will be denoted by $X'$, with norm $\|\cdot \|_{-1}$.
We define the skew-symmetric trilinear operator $b^*:X\times X \times X \rightarrow \mathbb{R}$ by
\begin{align*}
b^*(u, v, w):=\frac{1}{2}(u\cdot\nabla  v,  w) - \frac{1}{2} (u\cdot\nabla  w,  v),
\end{align*}
which satisfies 
\begin{align}\label{skewsymmest}
b^*(u, v, w)\leq M \|\nabla u\| \|\nabla v\| \|\nabla w\|,
\end{align}
for any $u,v,w\in X$, where $M$ is a constant depending on $|\Omega|$ only, see \cite{laytonbook}.

In our analysis, the following natural norm on  $(X,Q)$ arises
\begin{align}
\label{eqn:norm}
\|(v,q)\|_X \coloneqq \sqrt{\nu\|\nabla v\|^2+ \epsilon\|q\|^2}.
\end{align}

The FEM formulation of the steady NSE is given as follows: Find $(u,p)\in(X_h,Q_h)$ such that
\begin{equation}\label{weaknese}
\begin{aligned}
\nu(\nabla u,\nabla v)+b^*( u, u, v)-(p,\nabla\cdot v)&=( f, v),\\
(q,\nabla\cdot u)&=0,
\end{aligned}
\end{equation}
for all $( v,q)\in(X_h,Q_h).$ It is known that system \eqref{weaknese} has solutions for any data, and those solutions are unique if the small data condition $\kappa :=\nu^{-2}M\|f\|_{-1}<1$ is satisfied.  Moreover, all solutions to \eqref{weaknese} are bounded by $\|\nabla u\|\le \nu^{-1}\|f\|_{-1}.$ 

\begin{assumption}\label{assume:kappa}
	We will assume in our analysis that $\kappa<1$ and so that \eqref{weaknese} is well-posed.
\end{assumption}

\subsection{Discretization preliminaries}

We denote with $\tau_h$ a conforming, shape-regular, and simplicial triangulation of $\Omega$ with $h$ denoting the maximum element diameter of $\tau_h$. We represent the space of degree $k$ globally continuous piecewise polynomials on $\tau_h$ by $P_k(\tau_h)$, and $P_k^{disc}(\tau_h)$ the space of degree $k$ piecewise polynomials that can be discontinuous across elements. 

We choose the discrete velocity space by $X_h=X\cap P_k(\tau_h)$ and the pressure space $Q_h=\nabla\cdot X_h \subseteq Q$. With this choice of spaces, the discrete versions of \eqref{ns4}-\eqref{ns5} and \eqref{ns7} are equivalent, although in our computations we use only \eqref{ns7} and so the pressure space is never explicitly used. As discussed in \cite{HS18}, pressure recovery via the $L^2$ projection of $-\epsilon^{-1}  \sum\limits_{j=0}^{\infty} \nabla \cdot u_j$ into $Q\cap P_{k-1}(\tau_h)$ will yield a continuous and optimally accurate pressure. Under certain mesh structures, the $(X_h,Q_h)$ pair will satisfy the discrete inf-sup condition \cite{Z10, Z10a, NS15,arnold:qin:scott:vogelius:2D,GS19}.  While inf-sup is important for small $\epsilon$ in the IPP, our focus is on $\epsilon=1$ and so this compatibility condition is not necessary for our analysis to hold.

\subsection{Anderson acceleration}
Anderson acceleration is an extrapolation method used to improve convergence of fixed-point iterations.
Following  \cite{ToKe15,WaNi11,PRX19}, it may be stated as follows, where $Y$ is a normed vector space and $g: Y \rightarrow Y$.

\begin{alg}[Anderson iteration] \label{alg:anderson}
	Anderson acceleration with depth $m$ and damping factors $\beta_k$.\\ 
	Step 0: Choose $x_0\in Y.$\\
	Step 1: Find $w_1\in Y $ such that $w_1 = g(x_0) - x_0$.  
	Set $x_1 = x_0 + w_1$. \\
	Step $k$: For $k=2,3,\ldots$ Set $m_k = \min\{ k-1, m\}.$\\
	\indent [a.] Find $w_{k} = g(x_{k-1})-x_{k-1}$. \\
	\indent [b.] Solve the minimization problem for the Anderson coefficients $\{ \alpha_{j}^{k}\}_{j = 1}^{m_k}$
	\begin{align}\label{eqn:opt-v0}
	\textstyle \min 
	\left\| \left(1- \sum\limits_{j=1}^{m_k} \alpha_j^{k} \right) w_k + \sum\limits_{j = 1}^{m_k}  \alpha_j^{k}  w_{k-j} \right\|_Y.
	\end{align}
	\indent [c.] For damping factor $0 < \beta_k \le 1$, set
	\begin{align}\label{eqn:update-v0}
	\textstyle
	x_{k} 
	= (1-\sum\limits_{j = 1}^{m_k}\alpha_j^k) x_{k-1} + \sum\limits_{j= 1}^{m_k} \alpha_j^{k} x_{j-1}
	+ \beta_k \left(  (1- \sum\limits_{j= 1}^{m_k} \alpha_j^{k}) w_k + \sum\limits_{j=1}^{m_k}\alpha_j^k w_{k-j}\right),
	\end{align}
	where $w_{j} = g(x_{j-1}) - x_{j-1}$ is the nonlinear residual (and also sometimes referred to as the update step).
\end{alg}

Note that depth $m=0$ returns the original fixed-point iteration.  
We define the optimization gain factor $\theta_{k}$ by 
\begin{align}\label{eqn:thetadef}
\theta_k = \frac{ \left\| \left(1- \sum\limits_{j=1}^{m_k} \alpha_j^{k} \right) w_k + \sum\limits_{j = 1}^{m_k}  \alpha_j^{k}  w_{k-j} \right\|_Y } {\nrs{w_{k}}_Y},
\end{align}
representing the ratio gain of the minimization problem \eqref{eqn:opt-v0} using $m_k$ compared to the $m=0$ (usual fixed point iteration) case.
The gain factor $\theta_k$ plays a critical role in the general AA convergence theory \cite{EPRX20,PR21} 
that reveals how AA improves convergence: specifically, the acceleration reduces the first-order residual 
term by a factor of $\theta_k$, but introduces higher-order terms into the residual 
expansion.  

The next two assumptions give sufficient conditions
on the fixed point operator $g$ for the theory of \cite{PR21} to be applied.
\begin{assumption}\label{assume:g}
	Assume $g\in C^1(Y)$ has a fixed point $x^\ast$ in $Y$,
	and there are positive constants $C_0$ and $C_1$ with
	\begin{enumerate}
		\item $\nr{g'(x)}_Y \le C_0$ for all $x\in Y$, and 
		\item $\nr{g'(x) - g'(y)}_Y \le C_1 \nr{x-y}_Y$
		for all $x,y \in Y$.
	\end{enumerate}
\end{assumption}

\begin{assumption}\label{assume:fg} 
	Assume there is a constant $\sigma> 0$ for which the differences between consecutive
	residuals and iterates satisfy
	$$ \| w_{{k}+1} - w_{k}\|_Y  \ge \sigma \| x_{k} - x_{{k}-1} \|_Y, \quad {k} \ge 1.$$
\end{assumption}
Assumption \ref{assume:fg} is satisfied, for example, if $g$ is contractive (i.e. if $C_0<1$ in Assumption 2.2).  Other ways that the assumption is satisfied
are discussed in \cite{PR21}.  Under Assumptions \ref{assume:g} and \ref{assume:fg}, 
the following result from \cite{PR21} produces a bound on the residual $\nr{w_{k+1}}$ in terms
of the previous residuals.

\begin{theorem}[Pollock et al., 2021]  \label{thm:genm}
Let Assumptions \ref{assume:g} and \ref{assume:fg} hold, and suppose the direction sines between each column $j$ of  matrix 
\begin{align}
\label{mtxF}
F_j = \left( \begin{array}{cccccc}(w_{j}-w_{j-1}) & (w_{j-1} - w_{j-2}) &    \cdots & (w_{j-m_j+1} - w_{j-m_j}) \end{array} \right) = (f_{j,i}) 
\end{align}
and the subspace spanned by the preceeding columns satisfies $|\sin(f_{j,i},\text{span }\{f_{j,1},$ $\ldots, f_{j,i-1}\})| \ge c_s >0$, for $j = 1, \dots, m_k$. Then the residual $w_{k+1} = g(x_k)-x_k$ from Algorithm \ref{alg:anderson} (depth $m$) satisfies the bound
\begin{align}\label{eqn:genm}
\nr{w_{k+1}}_Y & \le \nr{w_k}_Y \Bigg(
\theta_k ((1-\beta_{k}) + C_0 \beta_{k})
+ \frac{C C_1\sqrt{1-\theta_k^2}}{2}\bigg(
\nr{w_{k}}_Y h(\theta_{k})
\nonumber \\ &
+ 2  \sum_{n = k-{m_{k}}+1}^{k-1} (k-n)\nr{w_n}_Y h(\theta_n) 
+ m_{k}\nr{w_{k-m_{k}}}_Yh(\theta_{k-m_{k}})
\bigg) \Bigg),
\end{align}
where  each $h(\theta_j) \le C \sqrt{1 - \theta_j^2} + \beta_{j}\theta_j$,
and $C$ depends on $c_s$ and the implied upper bound on the direction cosines. 
\end{theorem}

The estimate \eqref{eqn:genm} shows how the relative contributions from 
the lower and higher order terms are determined by the gain factor $\theta_k$:
the lower order terms are scaled by $\theta_k$ and the higher-order terms by $\sqrt{1 - \theta_k^2}$. The estimate reveals that while larger choices of $m$ generally provide lower $\theta_k$'s which reduces the lower order contributions to the residual, it also incurs a cost of both increased accumulation and weight of higher order terms. If recent residuals are small then greater algorithmic depths $m$ may be advantageous, but if not, large $m$ may slow or prevent convergence. As discussed in \cite{PR21}, this suggests that depth selection strategies that use small $m$ early in the iteration and large $m$ later may be advantageous in some settings.

This result supposes the sufficient linear independence of the columns of each matrix 
$F_j$ given by \eqref{mtxF}.  As discussed in \cite{PR21}, this assumption can be both verified and ensured, so
long as the optimization problem is solved in a norm 
induced by an inner-product.  One can safeguard 
by sufficiently reducing $m$ or by removing columns of $F_j$ where the desired inequality fails to hold, as demonstrated in \cite{PR21}.

\section{The iterated penalty Picard method and associated solution operator properties}
This section presents some properties of the IPP iteration and its associated fixed point function.
\subsection{Iterated penalty Picard method}
This subsection studies some properties of IPP method. We begin by defining its associated fixed point operator.
\begin{dfn} We define a mapping $G:(X_h,Q_h)\to (X_h,Q_h), \ G(u,p) = (G_1(u,p), G_2(u,p))$ such that for any $(v,q)\in (X_h,Q_h)$
	\begin{equation}\label{ipm2}
	\begin{aligned}
	\nu(\nabla G_1(u,p),\nabla v)+b^*( u, G_1(u,p), v)-(G_2(u,p),\nabla\cdot v)&=( f, v),\\
	\varepsilon(G_2(u,p),q)+(\nabla\cdot G_1(u,p),q) &=\varepsilon(p,q).
	\end{aligned}
	\end{equation}
\end{dfn}

Thus the IPP method for solving steady NSE can be rewritten now as follows.
\begin{alg}\label{alg:ipm}
	The iterated penalty method for solving steady NSE is
	\begin{enumerate}
		\item[Step 0] Guess $(u_0,p_0)\in (X_h,Q_h)$.
		\item[Step $k$] Find $(u_{k+1}, p_{k+1}) = G(u_k,p_k)$.
	\end{enumerate}
\end{alg}


We now show that $G$ is well-defined, and will then prove smoothness properties for it.

\begin{lemma}
	\label{lemma:uniqueness}
	The operator $G$ is well defined. Moreover,
	\begin{align}
	\label{utildebd}
	\|\nabla G_1(u,p)\| \le \nu^{-1}\|f\|_{-1} + \sqrt{\frac{\epsilon}\nu} \|p\|,
	\end{align} 
	for any $(u,p)\in (X_h,Q_h).$
\end{lemma}
\begin{proof}
	Given $f,u,p$, assume $(u_1,p_1),(u_2,p_2)\in (X_h,Q_h)$ are solutions to \eqref{ipm2}. Subtracting these two systems and letting $e_u=u_1-u_2$ and $e_p=p_1-p_2$ produces
	\begin{align*}
	\nu(\nabla e_u,\nabla v)+b^*( u, e_u , v)-(e_p ,\nabla\cdot v)&=0,\\
	\varepsilon(e_p,q)+(\nabla\cdot e_u,q) &=0.
	\end{align*}
	Setting $v=e_u$ and $q=e_p$, and adding these equations gives
	\begin{align*}
	\nu \|\nabla e_u\|^2 +\varepsilon\|e_p\|^2=0,
	\end{align*}
	which is satisfied if $e_u=e_p=0$ implying the solution of \eqref{ipm2} is unique.  Because \eqref{ipm2} is linear and finite dimensional, solutions must exist uniquely.  Choosing $v=G_1(u,p)$ and $q=G_2(u,p)$ in \eqref{ipm2} produces
	\begin{align*}
	\| G(u,p)\|_X^2 =\nu\|\nabla G_1(u,p)\|^2+ \epsilon\| G_2(u,p)\|^2 &\le \epsilon\| p\|^2 + \nu^{-1}\|f\|_{-1}^2, 
	\end{align*}
	thanks to Cauchy-Schwarz and Young's inequalities. 
	This shows the solution $G(u,p)$ is bounded continuously by the data, proving \eqref{ipm2} is well-posed and thus $G$ is well-defined.
	Additionally, dropping the term $\|G_2(u,p)\|^2$ and taking square root yields \eqref{utildebd}.
	
\end{proof}

\begin{lemma}
	\label{lemma:conv1}
	Under Assumption \ref{assume:kappa}, let $(u,p)$ be the solution of \eqref{weaknese} and $(u_k, p_{k})$ be $k^{th}$ iteration from Algorithm \ref{alg:ipm}.  Then we have
	\begin{align}
	\label{eqn:conv1}
	\|(u_{k+1}, p_{k+1}) - (u,p)\|_X <  \|(u_k,p_k) - (u, p)\|_X.
	\end{align}
\end{lemma}

\begin{proof}
	Subtracting equations \eqref{weaknese} from \eqref{ipm2} with $(u_{k+1}, p_{k+1})$ gives
	\begin{align}
	\label{e1}
	\nu (\nabla (u_{k+1} - u),\nabla v) + b^*(u_k, u_{k+1} - u, v) + b^*(u_k - u, u,v) \nonumber\\
	-(p_{k+1} - p,\nabla \cdot v) =& 0,\\
	\label{e2}
	\epsilon(p_{k+1} - p, q) + (\nabla \cdot (u_{k+1} - u),q) =& \epsilon(p_k - p,q).
	\end{align}
	Adding these equations together and setting $v = u_{k+1} - u, q = p_{k+1} - p$ produces
	\begin{multline*}
	\nu \| \nabla (u_{k+1} - u)\|^2 +\epsilon \| p_{k+1} - p\|^2 \\ \le M \|\nabla (u_k - u)\| \|\nabla u\| \| \nabla (u_{k+1} - u)\| + \epsilon \| p_{k}-p\| \|p_{k+1}-p\|,
	\end{multline*}
	thanks to \eqref{skewsymmest} and Cauchy-Schwarz inequality. Then, using  $\|\nabla u\|\le \nu^{-1}\|f\|_{-1}, $ and Young's inequality gives
	\begin{align*}
	\nu \| \nabla (u_{k+1} - u)\|^2 + \epsilon \| p_{k+1} - p\|^2 
	\le 
	\nu \kappa^{2} \|\nabla (u_k - u)\|^2
	+
	\epsilon \| p_{k}-p\|^2,
	\end{align*}
	where $\kappa:=M\nu^{-2}\|f\|_{-1}$. Thanks to the Assumption \ref{assume:kappa} and taking the square root on both sides gives \eqref{eqn:conv1}.
\end{proof}

Lemma \ref{lemma:conv1} shows us that Algorithm \ref{alg:ipm} converges when the small data condition $\kappa <1$ is satisfied. However, it tells us nothing when $\kappa \ge1$.  With Anderson acceleration, we can discuss the convergence behavior $\kappa \ge 1$, see Theorem \ref{thm:aa}.  Next, we show the solution operator $G$ is Lipschitz continuous and Fr\'echet differentiable.
\begin{lemma}\label{LipschitzG}
	For any $( u,p),\ ( w,z)\in (X_h, Q_h)$, we have
	\begin{align}
	\label{eqn:lip}
	\|G( u, p) - G( w, z)\|_X\leq C_L\|( u, p) - ( w, z)\|_X,
	\end{align}  
	where 
	$ C_L = \max \{1,\kappa + \sqrt{\epsilon/\nu^3}M\| p\| \}$. 
\end{lemma}
	\begin{remark}
		Equation \eqref{eqn:lip} tells us that Algorithm \ref{alg:ipm} converges linearly with rate $C_L,$ which may be larger than the usual Picard method's rate of $\kappa$ \cite{GR86}. This is not surprising, since (until now) IPP would never be used with small $\epsilon$.  In \cite{C93}, for example, a smaller $C_L$ is found for IPP, but small $\epsilon$ is assumed as is an inf-sup compatibility condition on the discrete spaces.  We will show in section 4 that when IPP is enhanced with AA, the effective linear convergence rate will be much smaller than $C_L$ even when $\epsilon=1$ and without an assumption of an inf-sup condition, and the resulting solver is demonstrated to be very effective in section 5.
	\end{remark}
\begin{proof}
	From \eqref{ipm2} with $(u,p)$ and $(w,z)$, we obtain
	\begin{align*}
	\nu(\nabla( G_1( u,p)- G_1( w,z)),\nabla v)
	+ b^*( u, G_1( u,p)- G_1( w,z), v)
	& \\
	+
	b^*( u- w, G_1( w,z), v)
	-
	( G_2( u,p)- G_2( w,z),\nabla\cdot v)
	&=0, \\*
	\varepsilon( G_2( u,p)- G_2( w,z),q)+(q,\nabla\cdot( G_1( u,p)- G_1( w,z)))&=\varepsilon(p-z,q).
	\end{align*}
	Adding these equations and choosing $ v= G_1( u,p)- G_1( w,z)$ and $q= G_2( u,p)- G_2( w,z)$ yields
	\begin{align*}
	&\nu\|\nabla( G_1( u,p)- G_1( w,z))\|^2+\varepsilon\| G_2( u,p)- G_2( w,z)\|^2
	\\
	&= \varepsilon(p-z, G_2( u,p)- G_2( w,z))-b^*( u- w, G_1( w,z), G_1( u,p)- G_1( w,z)) \\
	&\le  \epsilon \| p-z\| \| G_2( u,p)- G_2( w,z)\|  
	\\
	&+ M \| \nabla (u-w)\| \| \nabla (G_1(w,z))\| \|  \nabla (G_1( u,p)- G_1( w,z))\| ,
	\end{align*}
	thanks to Cauchy-Schwarz and \eqref{skewsymmest}.
	Applying Young's inequality provides
	\begin{multline*}
	\nu\|\nabla( G_1( u,p)- G_1( w,z))\|^2+\varepsilon\| G_2( u,p)- G_2( w,z)\|^2
	\\ \leq 
	\varepsilon \|p-z\|^2
	+
	\nu^{-1}M^2\|\nabla G_1(w,z)\|^2 \|\nabla( u- w)\|^2,
	\end{multline*}
	which reduces to \eqref{eqn:lip} due to \eqref{eqn:norm} and \eqref{utildebd}.
\end{proof}
Next, we define an operator $G'$ and then show that $G'$ is  the Fr\'echet derivative of operator $G$. 

\begin{definition}\label{defineG'}
	Given $( u,p)\in(X_h,Q_h)$, define an operator $G'( u,p;\cdot,\cdot):(X_h,Q_h)\rightarrow(X_h,Q_h)$ by
	$$G'( u,p; h,s):=(G'_1( u,p; h,s), G'_2( u,p; h,s))$$
	satisfying for all $( h,s)\in(X_h,Q_h)$ 
	\begin{equation}
	\begin{aligned}\label{diffG}
	\nu(\nabla G'_1( u,p; h,s),\nabla v)
	+
	b^*( h,G_1( u,p), v)
	+
	b^*( u,G'_1( u,p; h,s), &v)
	\\-
	( G'_2( u,p; h,s),\nabla\cdot v)
	=&
	0,
	\\
	\varepsilon( G'_2( u,p; h,s),q)
	+
	(q,\nabla\cdot G'_1( u,p; h,s))
	=&
	\varepsilon(s,q).
	\end{aligned}
	\end{equation}
\end{definition}
\begin{lemma} \label{lemma:wellG'} $G'$ is well-defined and is the Fr\'echet derivative of operator $G$ satisfying
	\begin{align}\label{wellG'}
	\|G'( u,p; h,s)\|_X
	\leq
	C_L
	\|(h,s)\|_X.
	\end{align}
	and
	\begin{align}\label{LipschitzG'}
	\|G'( u+ h,p+s; w,z)-G'( u,p; w,z)\|_X
	\leq
	\hat{C}_L \|( w,z)\|_X \|( h,s)\|_X 
	\end{align}
	where $C_L$ is defined in Lemma \ref{LipschitzG} and $\hat C_L = \sqrt{10} \nu^{-3/2}MC_L$.
\end{lemma}
\begin{proof}
	The proof consists of three parts. First, we show that $G'$ is well-defined and \eqref{wellG'} holds. 
	%
	%
	Adding equations in \eqref{diffG} and setting $ v=G'_1( u,p; h,s)$ and $q= G'_2( u,p; h,s)$ produces
	\begin{align*}
	\nu \|\nabla G'_1( u,p; h,s)\|^2
	+
	\varepsilon \| G'_2( u,p; h,s)\|^2
	&=
	\varepsilon(s, G'_2( u,p; h,s))\\
	&-
	b^*(h,G_1( u,p),G'_1( u,p; h,s)).
	\end{align*}
	Applying Cauchy-Schwarz, (\ref{skewsymmest}) and Young's inequalities gives
	\begin{align*}
	\nu \|\nabla G'_1( u,p; h,s)\|^2
	+
	\varepsilon &\| G'_2( u,p; h,s)\|^2
	\\
	&\leq
	\varepsilon \|s\|^2
	+
	\nu^{-1}M^2  \|\nabla h\|^2 \|\nabla G_1( u,p)\|^2\\
	 &\leq
	\varepsilon \|s\|^2
	+
	\nu^{-1} M^2 (\nu^{-1}\|f\|_{-1} + \sqrt{\frac{\epsilon }{\nu}} \|p\|)^2 \|\nabla h\|^2  \\
	&\leq
		C_L^2
		(
		\varepsilon \|s\|^2
		+
		\nu \|\nabla h\|^2
		),
	\end{align*}
	thanks to  (\ref{utildebd}), which leads to (\ref{wellG'}).
	Since the system (\ref{diffG}) is linear and finite dimensional, (\ref{wellG'}) is sufficient to conclude that (\ref{diffG}) is well-posed.  
	
	The second part shows that $G'$ is the Fr\'echet derivative of $G$. Denote $ \eta_1 = G_1(u+h,p+s) - G_1(u,p) - G_1'(u,p; h,s), \eta_2 = G_2(u+h,p+s) - G_2(u,p) - G_2'(u,p;h,s).$
	Subtracting the sum of \eqref{diffG} and \eqref{ipm2} from the equation \eqref{ipm2} with $(u+h, p+s)$ yields
		\begin{align*}
		\nu (\nabla \eta_1, \nabla v) + b^*(u, \eta_1, v) + b^*(h, G_1(u+h,p+s) - G_1(u,p), v)\\ - (\eta_2, \nabla \cdot v) = 0,\\ 
		\epsilon ( \eta_2, q) + (\nabla \cdot \eta_1,q)  =0.
		\end{align*}
	
	Setting $v = \eta_1,\ q = \eta_2$ produces
	\begin{align*}
	\|G( u+ h,p+s)-G(&u,p)-G'( u,p; h,s)\|_X^2 \\ &\leq \nu^{-1}M^2 \|\nabla h\|^2 \|\nabla (G_1(u+h,p+s) - G_1(u,p))\|^2 \\
	 &\le \nu^{-3}M^2C_L^2\|( h,s)\|_X^4,
	\end{align*}
	thanks to Young's inequality and \eqref{eqn:lip}. Thus we have verified that $G'$ is the Fr\'echet derivative of $G$.
	
	Lastly, we show $G'$ is Lipschitz continuous over $(X_h,Q_h)$. For $(u,p), (h,s),$ $ (w,z) \in (X_h,Q_h)$, letting $e_1:=G'_1( u+ h,p+s; w,z)-G'_1( u,p; w,z)$, $e_2:=G'_2( u+ h,p+s; w,z)-G'_2( u,p; w,z)$, and then subtracting \eqref{ipm2} with $G'(u,p;w,z)$ from \eqref{ipm2} with $G'(u+h, p+s; w,z)$ yields
	\begin{align*}
	\nu(\nabla e_1, \nabla v) + b^*(w,G_1(u+h,p+s) - G_1(u,p), v) + b^*(h, G'_1(u,p;w,z),&v) \\
	+ b^*(u+h, e_1,v) - (e_2, \nabla \cdot v) & =0,\\
	\epsilon (e_2, q) + (\nabla \cdot e_1, q) & =0.
	\end{align*}
	Adding these equations and setting $v = e_1, q = e_2$ eliminates the fourth term and gives us
	\begin{equation*}
	\begin{aligned}
	\nu\| e_1\|^2
	+
	\varepsilon&\|e_2\|^2\\
	&=
	-b^*( w,G_1( u+ h,p+s)-G_1( u,p),e_1)
	-
	b^*( h,G'_1( u,p; w,z),e_1)  \\
	&\le  M\|\nabla w\| \|\nabla (G_1( u+ h,p+s)-G_1( u,p)\| \|\nabla e_1\|
	\\ &+ M\|\nabla h\|\|\nabla G'_1(u,p;w,z)\| \|\nabla e_1\| ,
	\end{aligned}
	\end{equation*}
	thanks to \eqref{diffG}.
	Now, applying the Young's inequality and \eqref{wellG'}, we get
	\begin{equation*}
	\begin{aligned}
	\|(e_1,e_2)\|^2_X
	&\leq 10\nu^{-3}M^2C_L^2  \|( w,z)\|^2_X \|( h,s)\|^2_X ,
	\end{aligned}
	\end{equation*}
	which implies to \eqref{LipschitzG'} after taking square roots on both sides. This finishes the proof.
\end{proof}
	Lastly, we show $G$ satisfies Assumption \ref{assume:fg}.
	\begin{lemma}
		\label{lemma1}
		The following inequality holds
		\begin{multline}
		\| \big( G(u_k, p_k) - (u_k, p_k) \big) -\big( G(u_{k-1}, p_{k-1}) - (u_{k-1}, p_{k-1}) \big)\|_X \\ \ge (1-C_L) \| (u_k ,p_k) -( u_{k-1}, p_{k-1}) \|_X ,
		\end{multline}
	\end{lemma}
	\begin{proof}
		If $C_L=1$, the result holds trivially.  Otherwise, $0\le C_L<1$, and from Lemma \ref{lemma:conv1}, we have 
		\begin{align*}
		& \| \big( G(u_k, p_k) - (u_k, p_k) \big) -\big( G(u_{k-1}, p_{k-1}) - (u_{k-1}, p_{k-1}) \big)\|_X \\
		&=  \| \big( G(u_k, p_k) - G(u_{k-1}, p_{k-1}) \big)- \big( (u_k,p_k) -  (u_{k-1}, p_{k-1}) \big)\|_X \\
		&\ge (1-C_L) \| (u_k,p_k) -  (u_{k-1}, p_{k-1})\|_X.
		\end{align*}
	\end{proof}

\section{The Anderson accelerated iterated penalty Picard scheme and its convergence}

Now we present the Anderson accelerated iterated penalty Picard (AAIPP) algorithm and its convergence properties.  Here, we continue the notation from section 3 that $G$ is the IPP solution operator for a given set of problem data.
\begin{alg}[AAIPP]
	\label{alg:aa}
	The AAIPP method with depth $m$ for solving the steady NSE is given by: 
	\begin{enumerate}
		\item[Step 0] Guess $(u_0,p_0)\in (X_h,Q_h)$.
		\item[Step 1] Compute $(\tilde u_1,\tilde p_1) = G(u_0,p_0)$\\
		Set residual $(w_1,z_1) = (\tilde u_1-u_0, \tilde p_1- p_0)$ and $(u_1,p_1) = (\tilde u_1,\tilde p_1) $.
		\item[Step $k$] For $k = 2,3,\dots$ set $m_k = \min\{k,m\},$
		\begin{enumerate}
			\item[a)] Find $(\tilde u_{k},\tilde p_k) = G(u_{k-1},p_{k-1})$ and set $(w_k,z_k) = (\tilde u_k-u_{k-1}, \tilde p_k- p_{k-1}).$
			\item[b)] Find $\{\alpha_j^k\}_{j=0}^{m_k}$ minimizing 
			\begin{align}
			\min_{\sum_{j=0}^{m_k}\alpha_{j}^{k}=1}\left\|  \sum\limits_{j=1}^{m_k}\alpha_j^k (w_{k-j}, z_{k-j})\right\|_X.
			\end{align}
			\item[c)] Update $(u_k,p_k) = (1-\beta_k)\left(  \sum\limits_{j=0}^{m_k} \alpha_j^k(u_j,p_j)  \right)+ \beta_k  \left( \sum\limits_{j=0}^{m_k} \alpha_j^k(\tilde u_j,\tilde p_j)  \right) $
			where $0<\beta_k\le 1$ is the damping factor.
		\end{enumerate}
	\end{enumerate}
\end{alg}
For any step $k$ with $\alpha_{m}^k =0,$ one should decrease $m$ and repeat Step k, to avoid potential cyclic behavior.  Under the assumption that $\alpha_m^k \neq 0$ and together with Lemmas \ref{LipschitzG}, \ref{lemma:wellG'} and \ref{lemma1}, we can invoke Theorem \ref{thm:genm} to establish the following convergence theory for AAIPP.
	\begin{thm}
		\label{thm:aa}
		For any step $k>m$ with $\alpha_{m}^k \neq 0$, the following bound holds for the AAIPP residual
		\begin{align*}
		\| (w_{k+1}, z_{k+1})\|_X \le& \theta_k(1-\beta_k + \beta_kC_L) \| (w_k,z_k)\|_X \\
		&+C\sqrt{1-\theta_{k}^2}\|(w_k,z_k)\|_{X} \sum\limits_{j=1}^{m} \|(w_{k-j+1},z_{k-j+1})\|_X,
		\end{align*}
		for the residual $(w_k,z_k)$ from Algorithm \ref{alg:aa}, where $\theta_k$ is the gain from the optimization problem,  $C_L$ is the Lipschitz constant of $G$ defined in Lemma \ref{LipschitzG}, and $C$ depending on $\theta_k, \beta_k, C_L$.
	\end{thm}
	This theorem tells us that Algorithm \ref{alg:aa} converges linearly with rate $\theta_k(1-\beta_k + \beta_kC_L) <1$, which improves on Algorithm \ref{alg:ipm} due to the scaling $\theta_k$ and the damping factor $\beta_k$.

\section{Numerical tests}
We now test AAIPP on the benchmark problems of 2D driven cavity, 3D driven cavity, and (time dependent) Kelvin-Helmholtz instability.
For all tests, the penalty parameter is chosen as $\epsilon=1$, and we use the velocity only formulation \eqref{ns7} for the IPP/AAIPP iteration. For the driven cavity problems, AA provides a clear positive impact, reducing total iterations and enabling convergence at much higher $Re$ than IPP without AA. For Kelvin-Helmholtz, AA significantly reduces the number of iterations needed at each time step.  Overall, our results show that AAIPP with $\epsilon=1$ is an effective solver.  In all of our tests, we use a direct linear solver (i.e. MATLAB's backslash) for the linear system solves of IPP/AAIPP as the problem sizes are such that direct solvers are more efficient.  Since the velocity-only formulation is used, convergence is measured in the $L^2$ norm instead of the $X$-norm, which requires the pressure.  For the problems we consider (up to 1.3 million degree of freedom in 3D) this remains a very robust and efficient linear solver.  In all of our tests, the cost of applying AA was negligible compared with the linear solve needed at each iteration, generally at least two orders of magnitude less.

\subsection{2D driven cavity}
We first test AAIPP for the steady NSE on a lid-driven cavity problem. The domain of the problem is the unit square $\Omega=(0,1)^2$ and we impose Dirichlet boundary conditions by $ u |_{y=1}=(1,0)^T $ and $u=0$ on the other three sides. The discretization uses $P_2$ elements on barycenter-refinement of uniform triangular mesh. We perform AAIPP with $m=0\ \text{(no acceleration)}, 1, 2, 5$ and $10$ with varying $Re,$ all with no relaxation ($\beta=1$).   In the following tests, convergence was declared if the velocity residual fell below $10^{-8}$.

Convergence results for AAIPP using $Re=1000, 5000$ and $10000$ with varying $m$ when $h=1/128$ are shown in Figure \ref{IP_conv}.  It is observed that as $m$ increases, AA improves convergence.  In particular, as $Re$ increases, AA is observed to provide a very significant improvement; for $Re=10000$, the iteration with $m=0$ (i.e. IPP iteration) fails but with AA convergence is achieved quickly.

In Figure \ref{Picard_conv}, convergence results by Anderson accelerated Picard (AAPicard) iteration are shown (i.e. no iterated penalty, and solve the nonsymmetric saddle point linear system at each iteration).  The same mesh is used, and here with $(P_2,P_1^{disc})$ Scott-Vogelius elements.  Convergence behavior for AAPicard is observed to be overall similar to that of AAIPP with $\epsilon=1$, with the exception that for lower $m=1,2$ AAPicard performs slightly better than AAIPP in terms of total number of iterations.  Since each iteration of AAIPP is significantly cheaper than each iteration of AAPicard (generally speaking, since AAPicard must solve a more difficult linear system), these results show AAIPP with $\epsilon=1$ performs very well.

%
%

\begin{figure}[H]
	\centering
	\includegraphics[width = .32\textwidth, height=.25\textwidth]{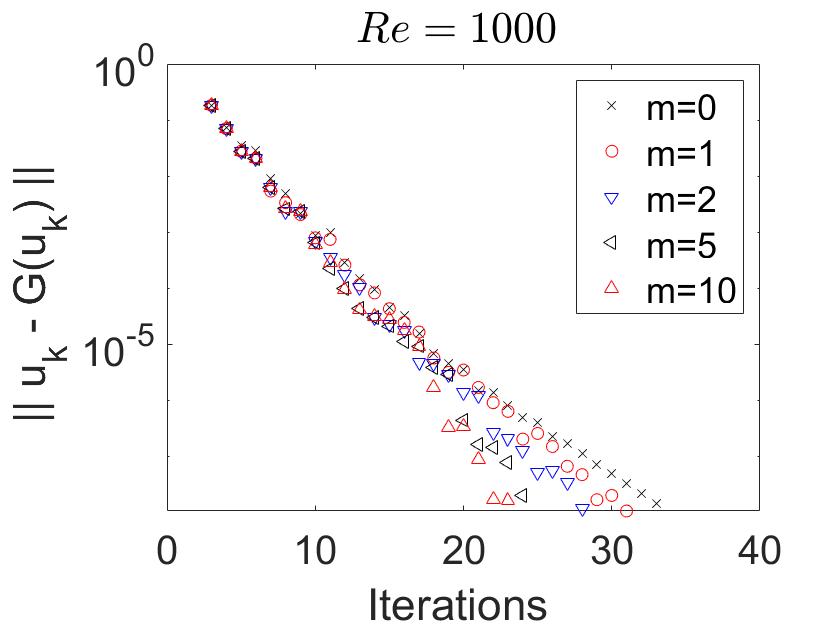}
	\includegraphics[width = .32\textwidth, height=.25\textwidth]{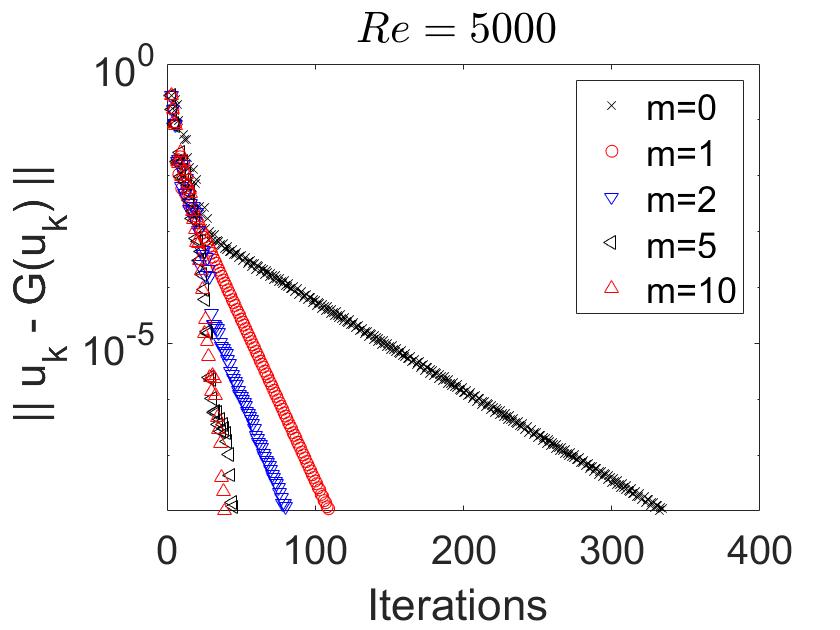}
	\includegraphics[width = .32\textwidth, height=.25\textwidth]{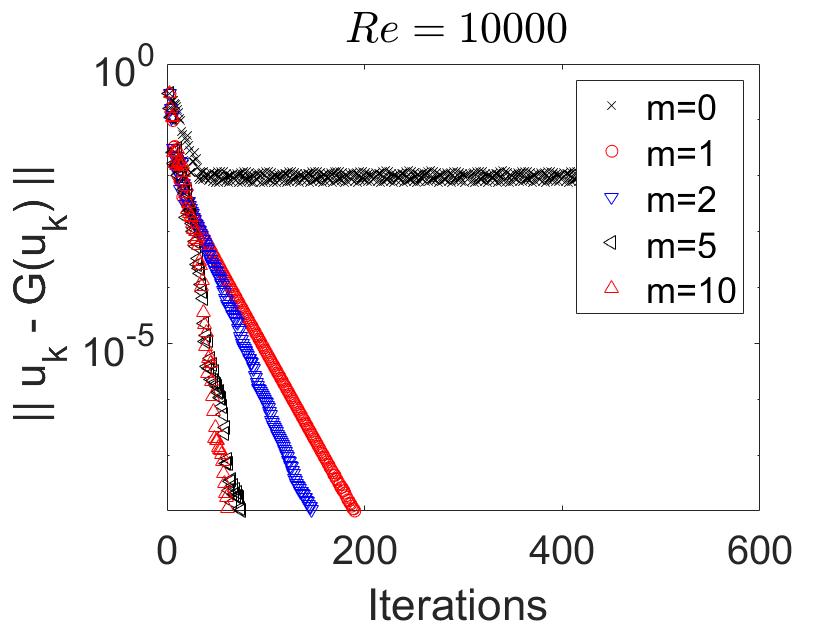}
	\caption{Shown above is convergence of AAIPP with varying $Re$ and $m$, using mesh width $h=1/128$.}\label{IP_conv}
\end{figure}

\begin{figure}[H]
	\centering
	\includegraphics[width = .32\textwidth, height=.25\textwidth]{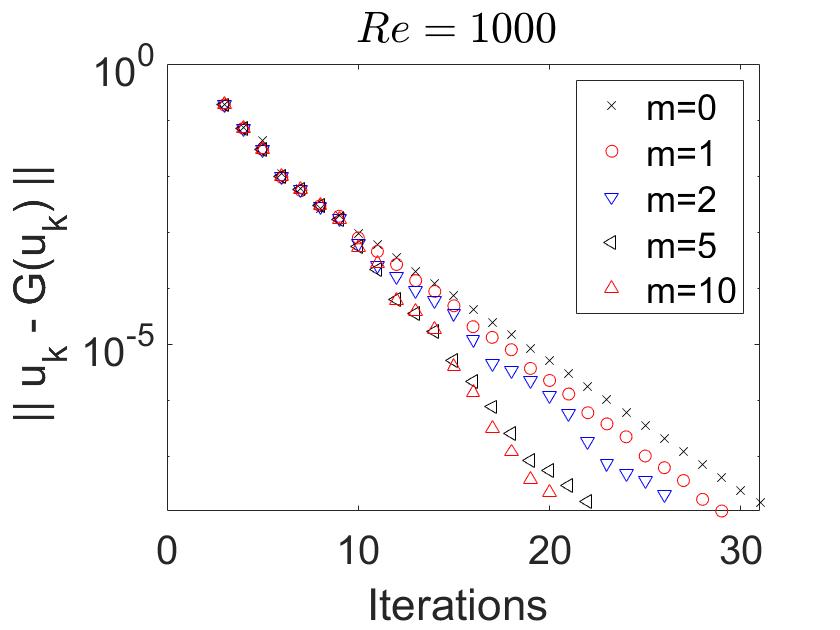}
	\includegraphics[width = .32\textwidth, height=.25\textwidth]{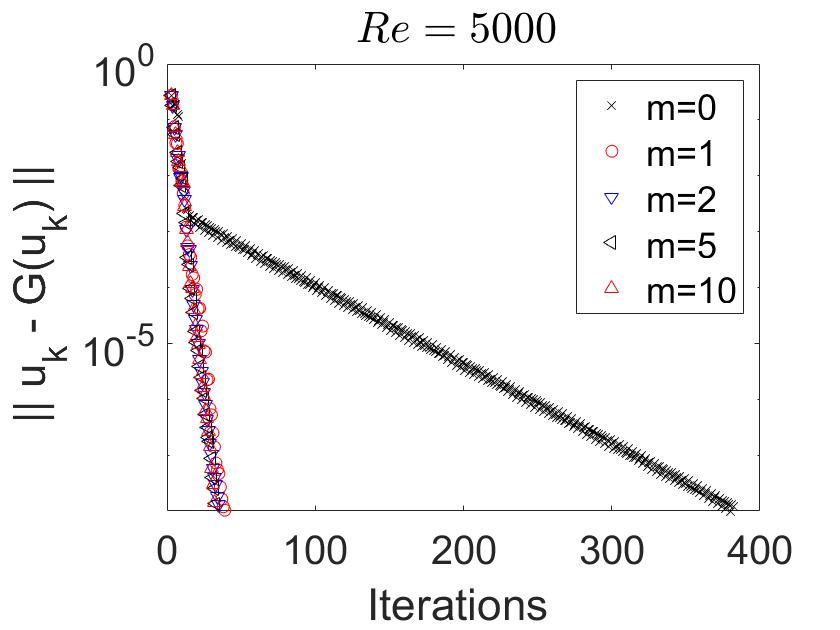}
	\includegraphics[width = .32\textwidth, height=.25\textwidth]{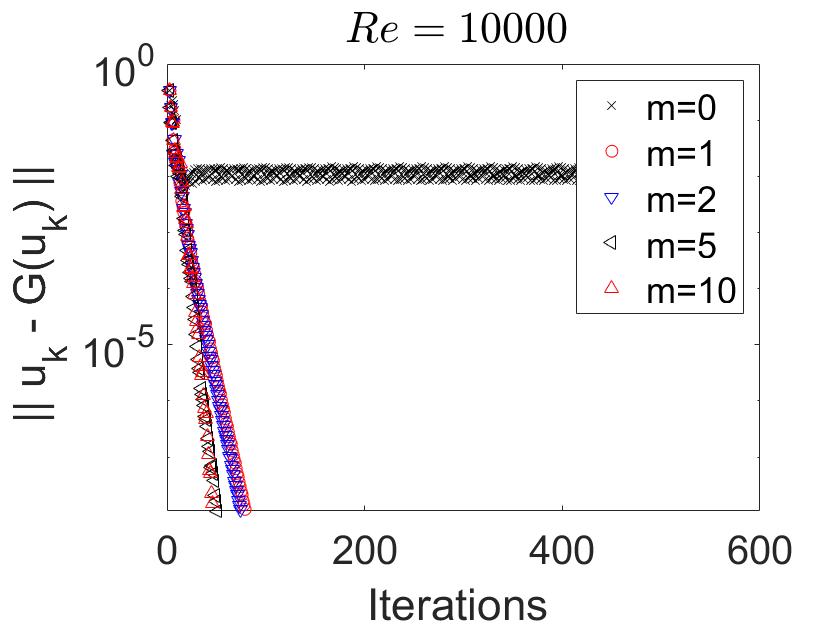}
	\caption{Shown above is convergence of AAPicard with varying $Re$ and $m$, using mesh width $h=1/128$.}\label{Picard_conv}
\end{figure}

\subsection{3D driven cavity}
We next test AAIPP on the 3D lid driven cavity. In this problem, the domain is the unit cube, there is no forcing $ (f = 0) $, and homogeneous Dirichlet boundary conditions are enforced on all walls and $u=\langle 1,0,0\rangle$ on the moving lid. We compute with $P_3$ elements on Alfeld split tetrahedral meshes with 796,722 and 1,312,470 total degrees of freedom (dof) that are weighted towards the boundary by using a Chebychev grid before tetrahedralizing. We test AAIPP with varying $Re$, $m$, and relaxation parameter $\beta$.

Figure \ref{midslideplanes} shows a visualization of the computed solutions by AAIPP with $m=10$ for $Re=100, 400$ and $1000$ on the 796,722 total dof, which are in good qualitative agreement with reference results of Wong and Baker \cite{WongBaker2002}. In Figure \ref{centerlines}, we compare centerline x-velocities for varying $Re=100, 400$ and $1000$ on the same mesh with reference data of Wong and Baker \cite{WongBaker2002} and obtain excellent agreement.

In table \ref{lowerRe}, the number of AAIPP iterations required for reducing the velocity residual to fall below $10^{-8}$, for varying dof, $Re$, relaxation parameter $\beta$ and depth, within 300 iterations. In each cases, we observe that as $m$ increases, the number of iterations decreases, and the maximum $m=k-1$ is observed to be the best choice in all cases. Also, the relaxation parameter $\beta$ is observed to give improved results for larger $Re$.  Thus, AAIPP with properly chosen depth and relaxation can significantly improve the ability of the iteration to to converge.

We also tested AAIPP with $Re=1500,2000,2500$ which would not converge within 1000 iterations without AA. Table  \ref{higherRe} shows that with sufficiently large $m$ and properly chosen relaxation parameter, AAIPP converges for even higher $Re$.  This is especially interesting since the bifurcation point where this problem becomes time dependent is around $Re\approx 2000$ \cite{CPHGG16}, and so here AAIPP is finding steady solutions in the time dependent regime.


\begin{table}[H]
	\centering
	\begin{tabular}{|l|l|l|cccccc|}
		\hline
		\multicolumn{3}{|c|}{} & \multicolumn{6}{|c|}{Iterations for convergence}\\
		\hline
		DoF       & $Re$& $\beta$ &   m=0  &    m=1 &    m=2 &    m=5 & m=10& m=k-1\\
		\hline
		796,722   &100  &  1      & $>300$ & $>300$ &  138   &  98    & 92  &75\\
		\hline  
		796,722   &400  &  1      & $>300$ & $>300$ &  260   &  120   & 91  &62\\
		\hline  
		796,722   &1000 &  1      & $>300$ & $>300$ & $>300$ & $>300$ & 126 &77\\
		\hline
		796,722   &1000 &  0.5    & $>300$ & $>300$ & $>300$ &  198   & 114 &68\\
		\hline
		796,722   &1000 & 0.2     & $>300$ & $>300$ & $>300$ &$>300$ & 177 & 87\\
		\hline
		1,312,470 &1000 &  1      & $>300$ & $>300$ & $>300$ & $>300$ & 140 &83\\
		\hline
		1,312,470 &1000 &  0.5    & $>300$ & $>300$ & $>300$ & 203    & 109 &71\\
		\hline
		1,312,470 &1000 &  0.2    & $>300$ & $>300$ & $>300$ & $>300$  & 174 &88\\
		\hline      
	\end{tabular}\caption{Shown above is the number of AAIPP iterations required for convergence in the 3D driven-cavity tests, with varying dof, $Re$, damping factor $\beta$ and depth. }\label{lowerRe}
\end{table}

\begin{figure}[H]
	\centering
	\includegraphics[width = .32\textwidth, height=.25\textwidth]{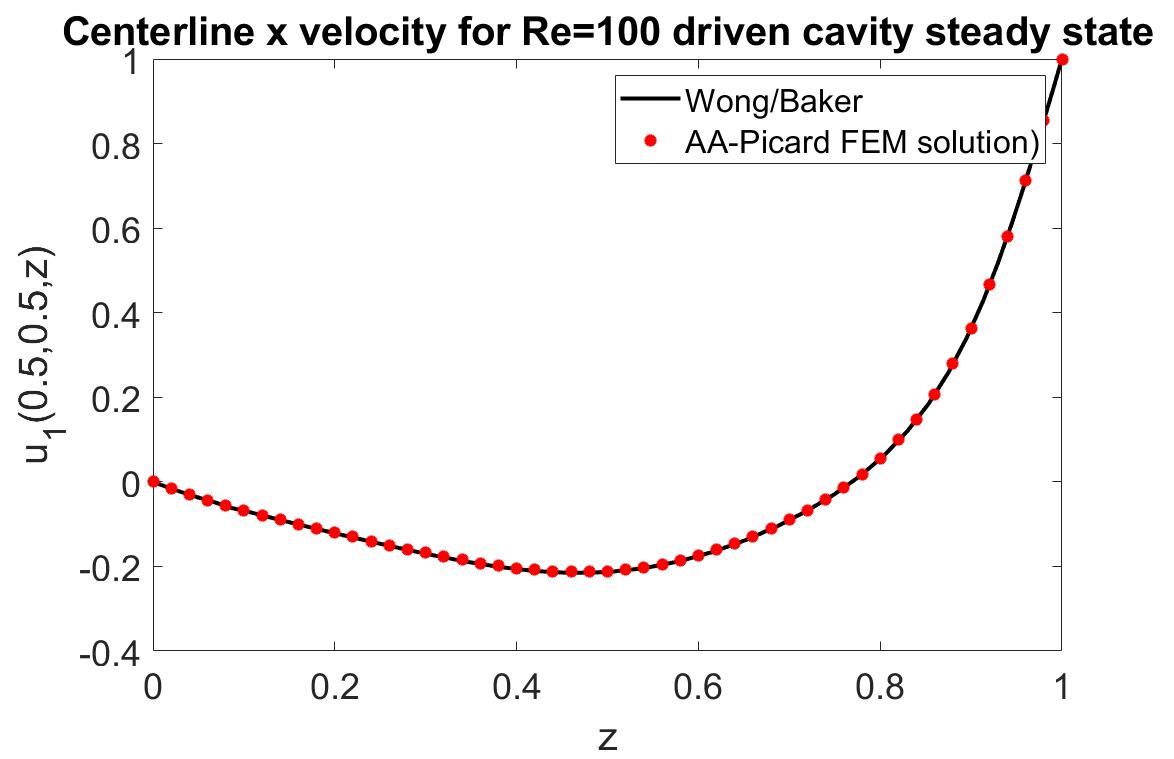}
	\includegraphics[width = .32\textwidth, height=.25\textwidth]{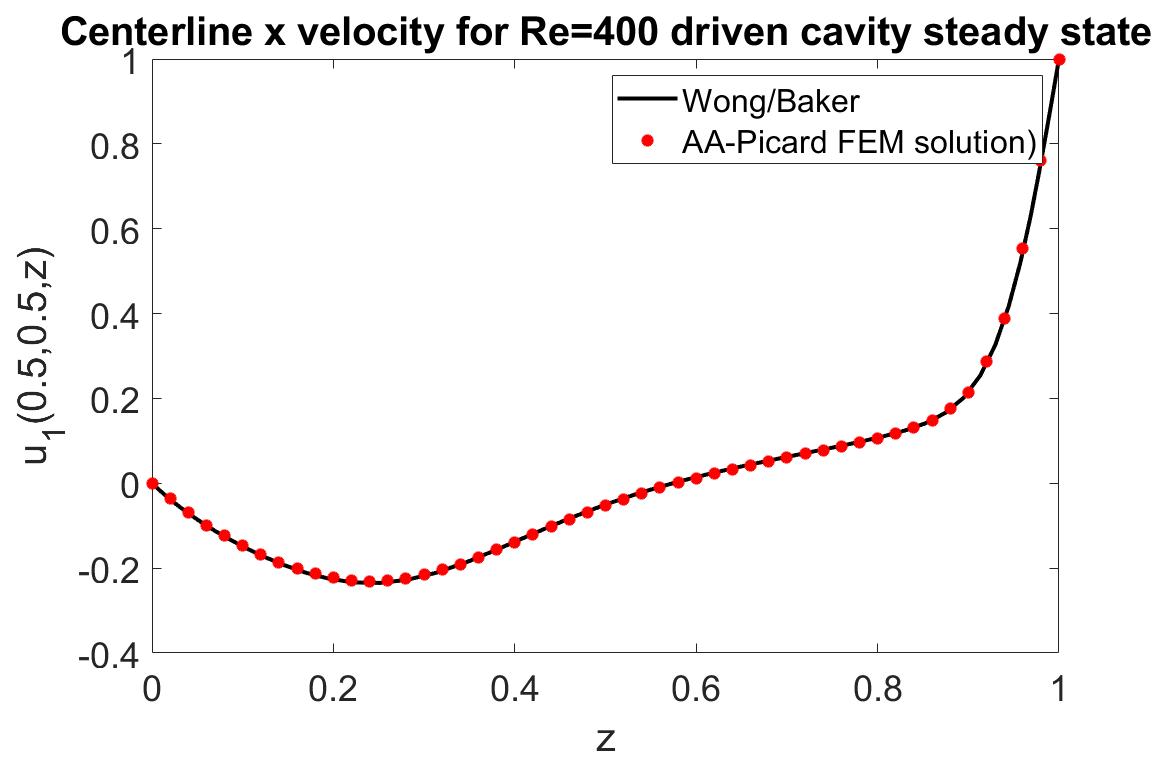}
	\includegraphics[width = .32\textwidth, height=.25\textwidth]{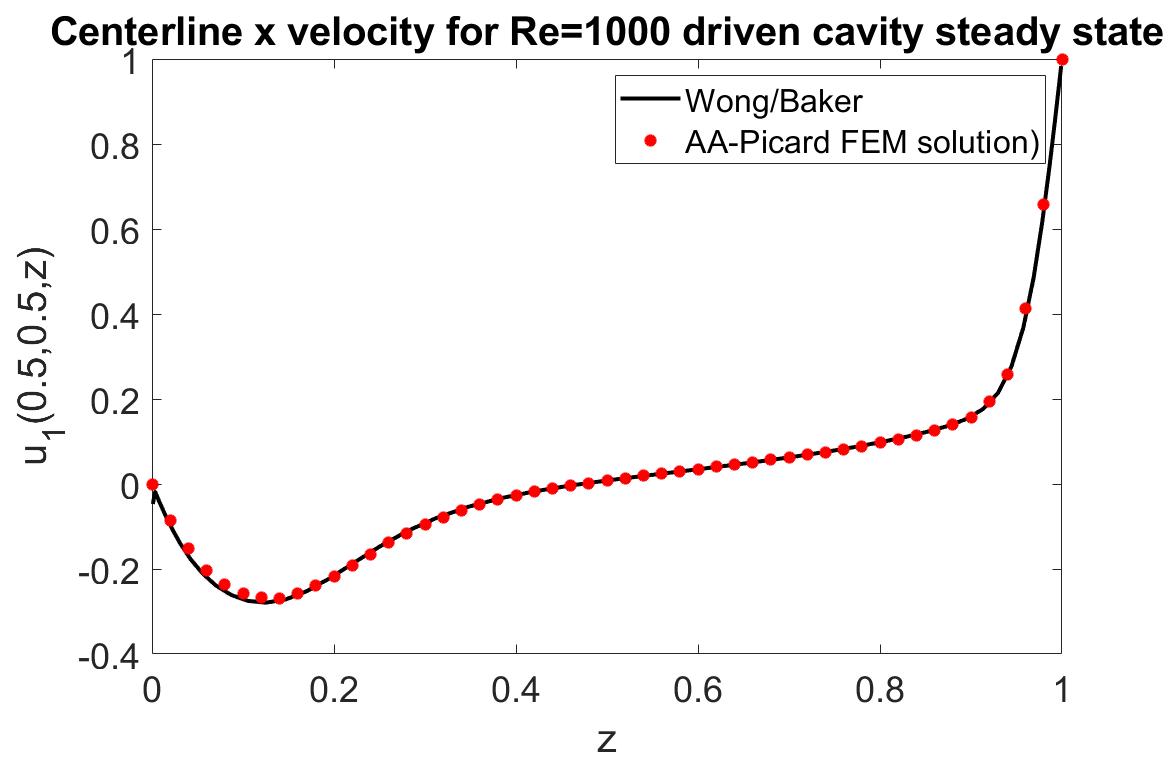}
	\caption{Shown above is the centerline x-velocity plots for the 3D driven cavity simulations at $Re=100,400,1000$, using AAIPP with $m=10$, $\beta=1$ and 796,722 dof.}\label{centerlines}
\end{figure}

\begin{table}[H]
	\centering
	\begin{tabular}{ccccc}
		\hline
		\hline
		$Re$ & dof &m & $\beta$ & k \\
		\hline
		\hline
		1500 &1,312,470 &k-1 & 0.5 & 102 \\
		1500 &1,312,470 &k-1 & 0.3 &111  \\
		\hline   
		2000 &1,312,470 &k-1 & 0.5 & 154 \\
		2000 &1,312,470 &k-1 & 0.3 & 162  \\
		\hline
		2500 &1,312,470& 100 & 0.5 & $ 381 $\\
		2500 &1,312,470& 100 & 0.3 & $ >1000 $\\
		2500 &1,312,470& k-1 & 0.5 &$ 371 $ \\
		\hline	
		\hline		
	\end{tabular}\caption{Shown above is the number of AAIPP iterations required for convergence in the 3D driven-cavity tests, with varying dof, $Re$, damping factor $\beta$ and depth of AA iterations. }\label{higherRe}
\end{table}

%
%

\begin{figure}[H]
	\centering
	\includegraphics[width = 0.9\textwidth, height=.23\textwidth]{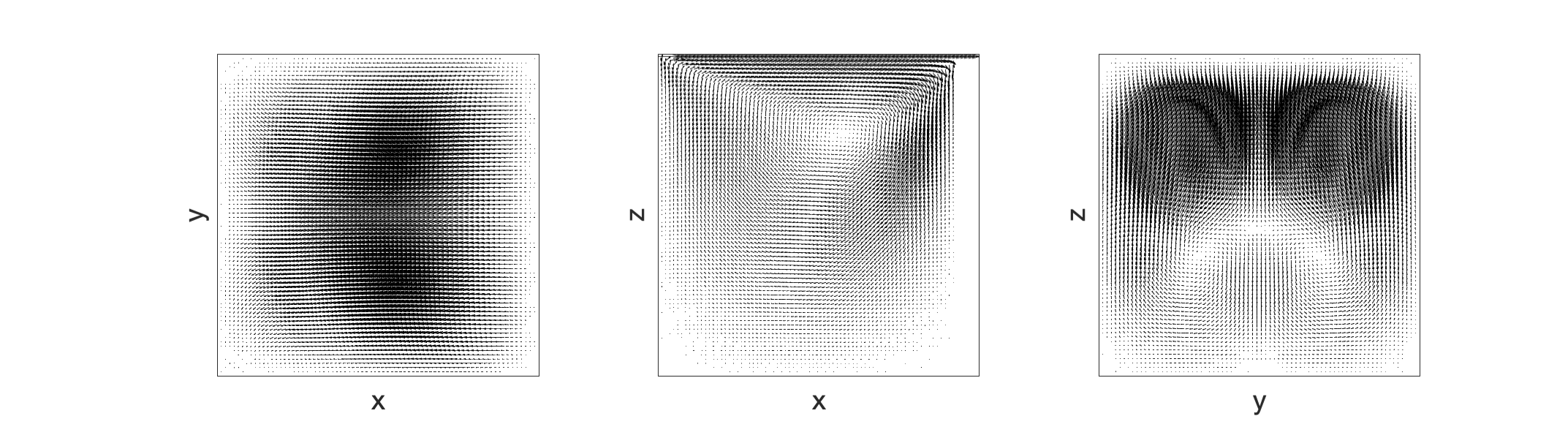}\\
	\includegraphics[width = 0.9\textwidth, height=.23\textwidth]{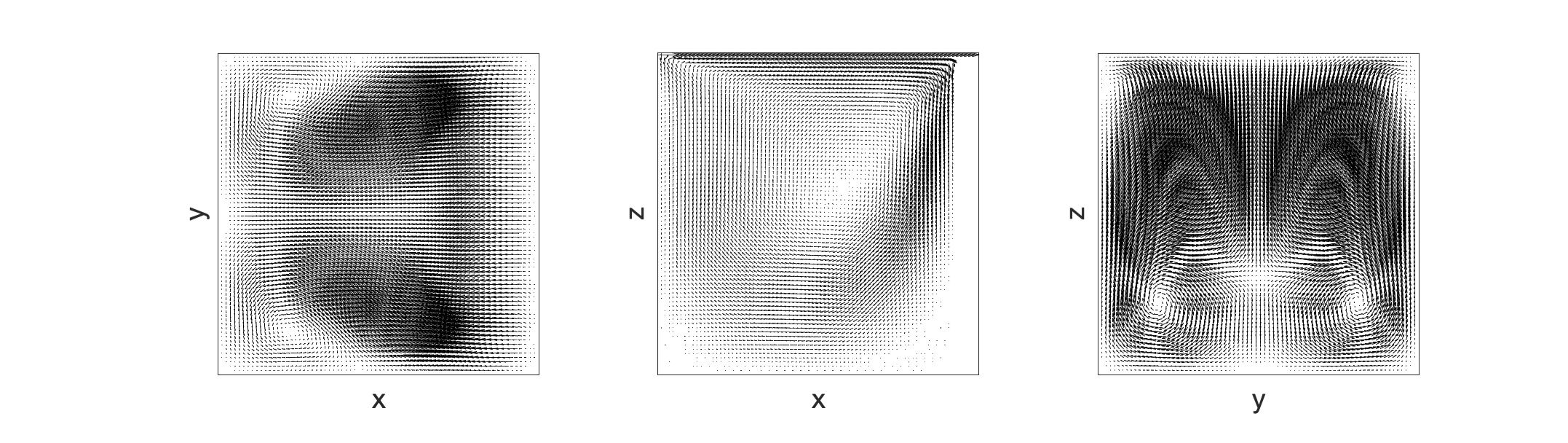}\\
	\includegraphics[width = 0.9\textwidth, height=.23\textwidth]{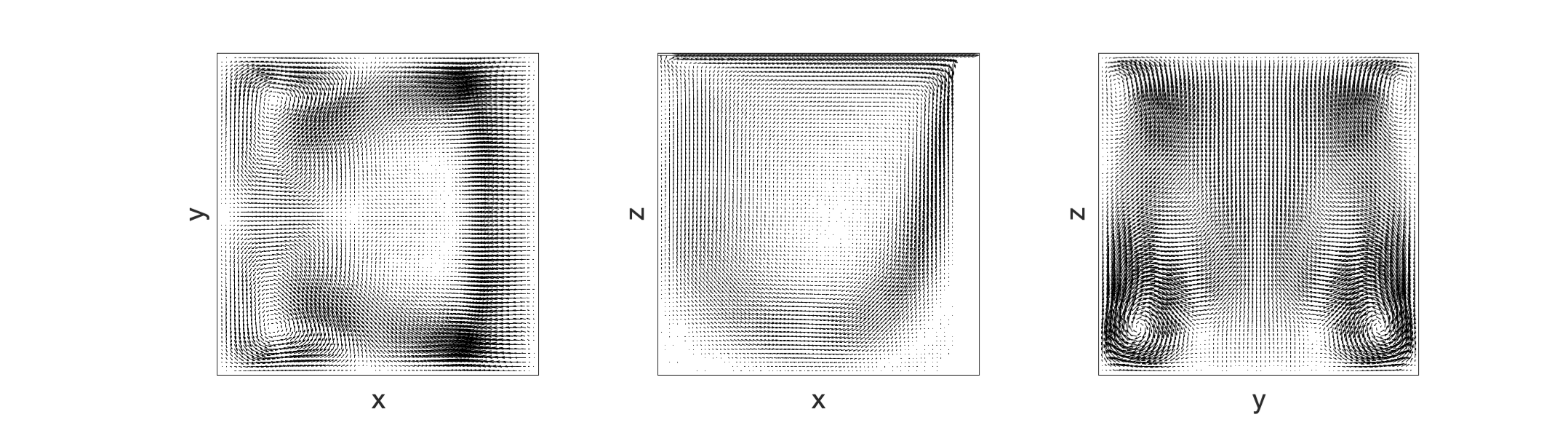}
	\caption{Shown above are the midsliceplane plots for the 3D driven cavity simulations at $ Re = 100, 400 $ and $1000$ by AAIPP with $m=10$, $\beta=1$ and 796,722 dof.}\label{midslideplanes}
\end{figure}

\subsection{Kelvin-Helmholtz instability}

\begin{figure}[!h]
	\begin{center}
		\includegraphics[width=.3\textwidth, height=.23\textwidth,viewport=65 40 520 390, clip]{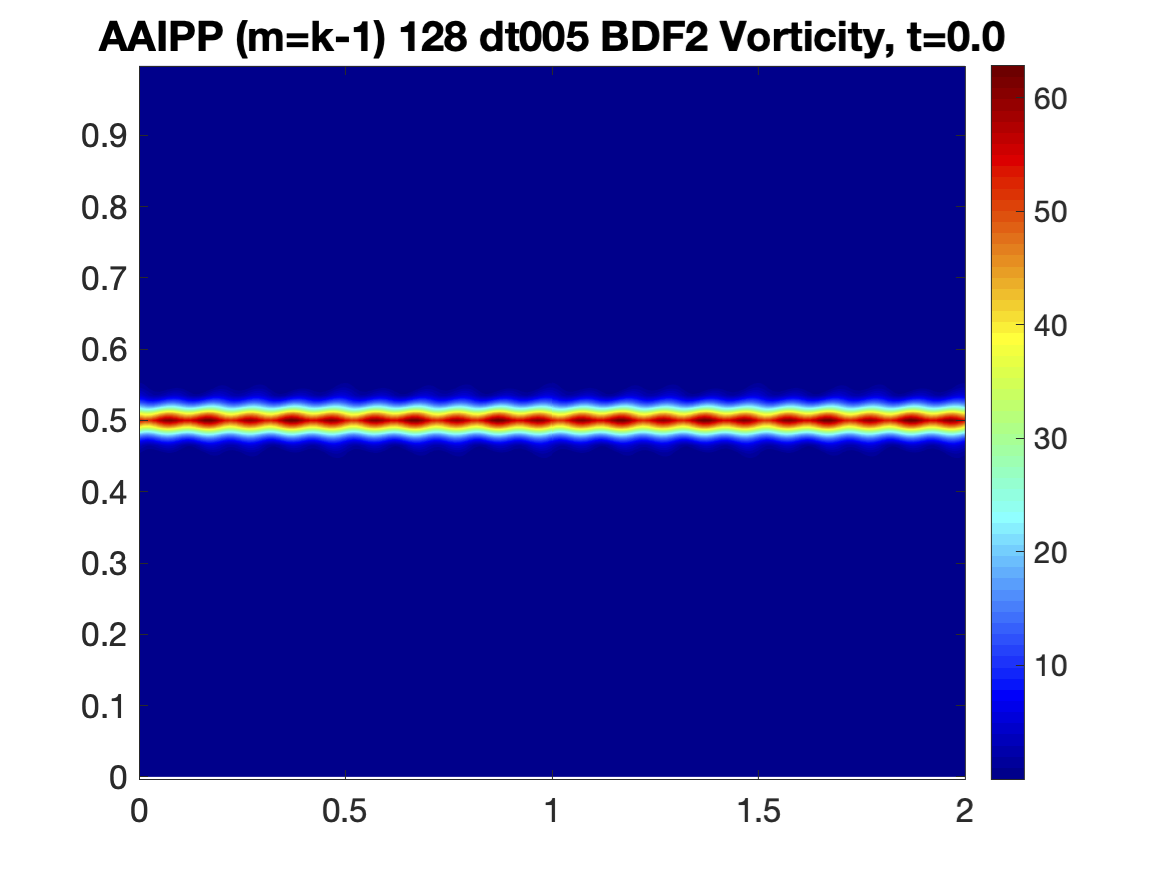}
		\includegraphics[width=.3\textwidth, height=.23\textwidth,viewport=65 40 520 390, clip]{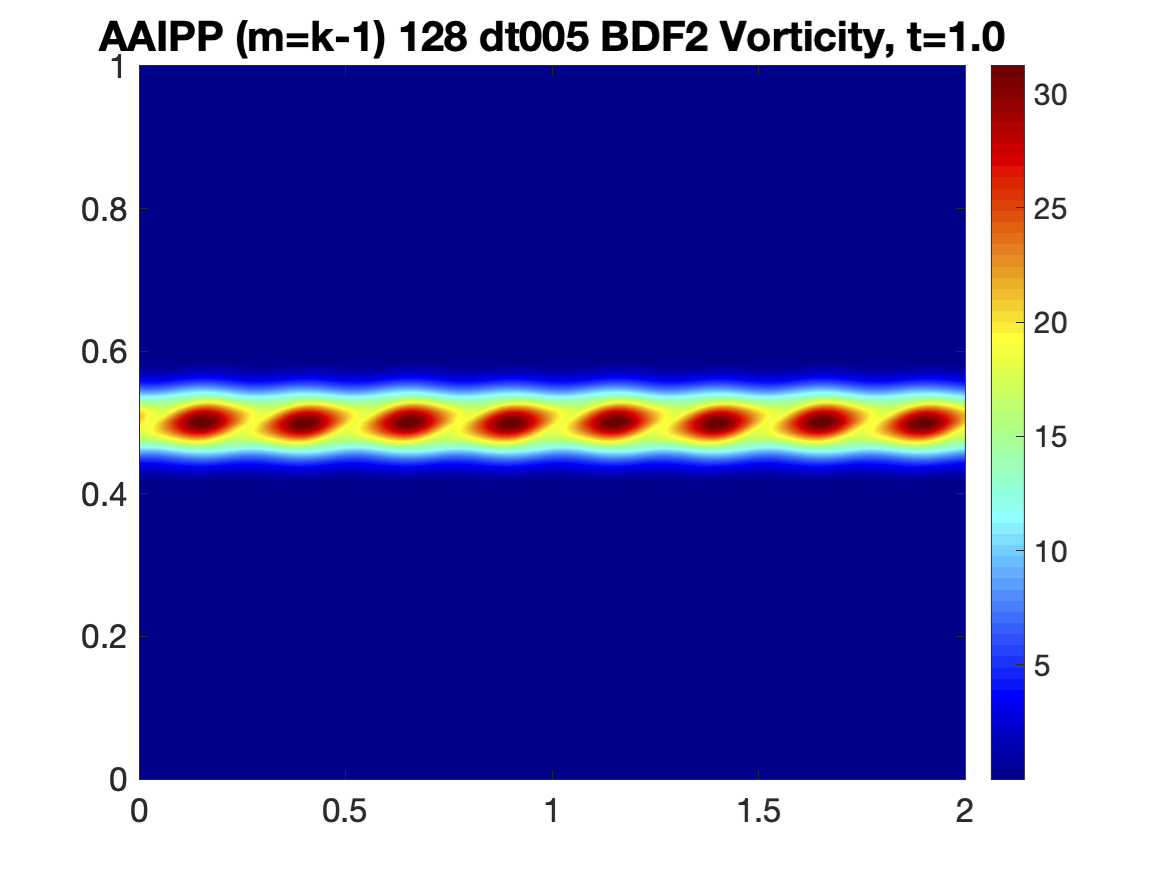} 
		\includegraphics[width=.3\textwidth, height=.23\textwidth,viewport=65 40 520 390, clip]{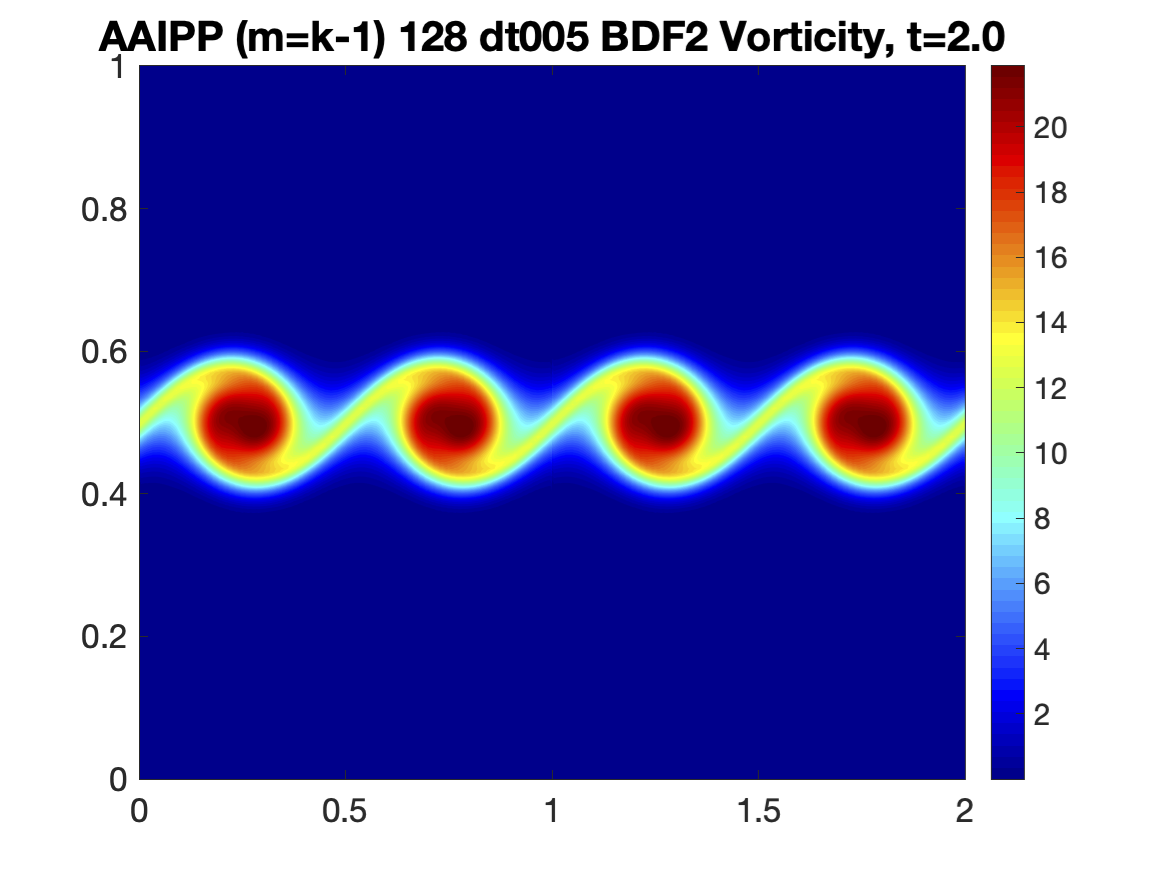}\\
		\includegraphics[width=.3\textwidth, height=.23\textwidth,viewport=65 40 520 390, clip]{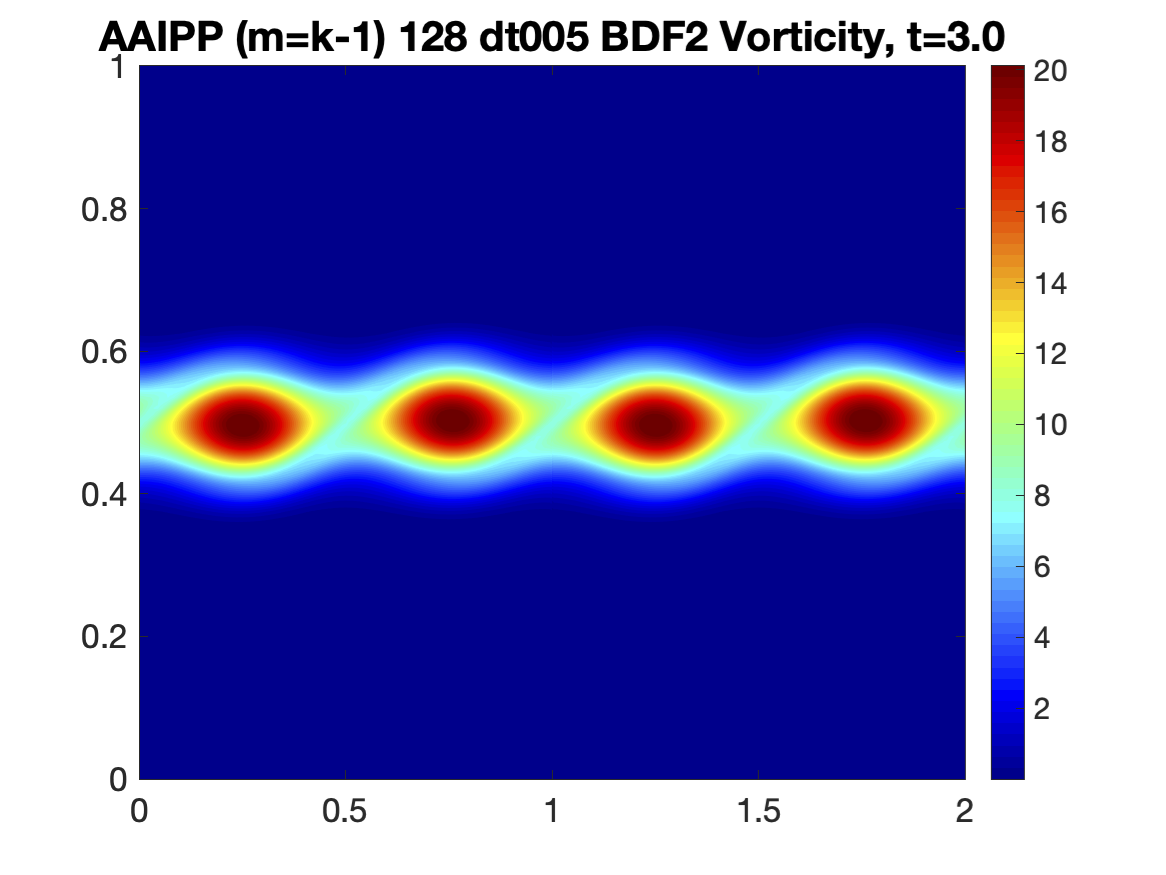}
		\includegraphics[width=.3\textwidth, height=.23\textwidth,viewport=65 40 520 390, clip]{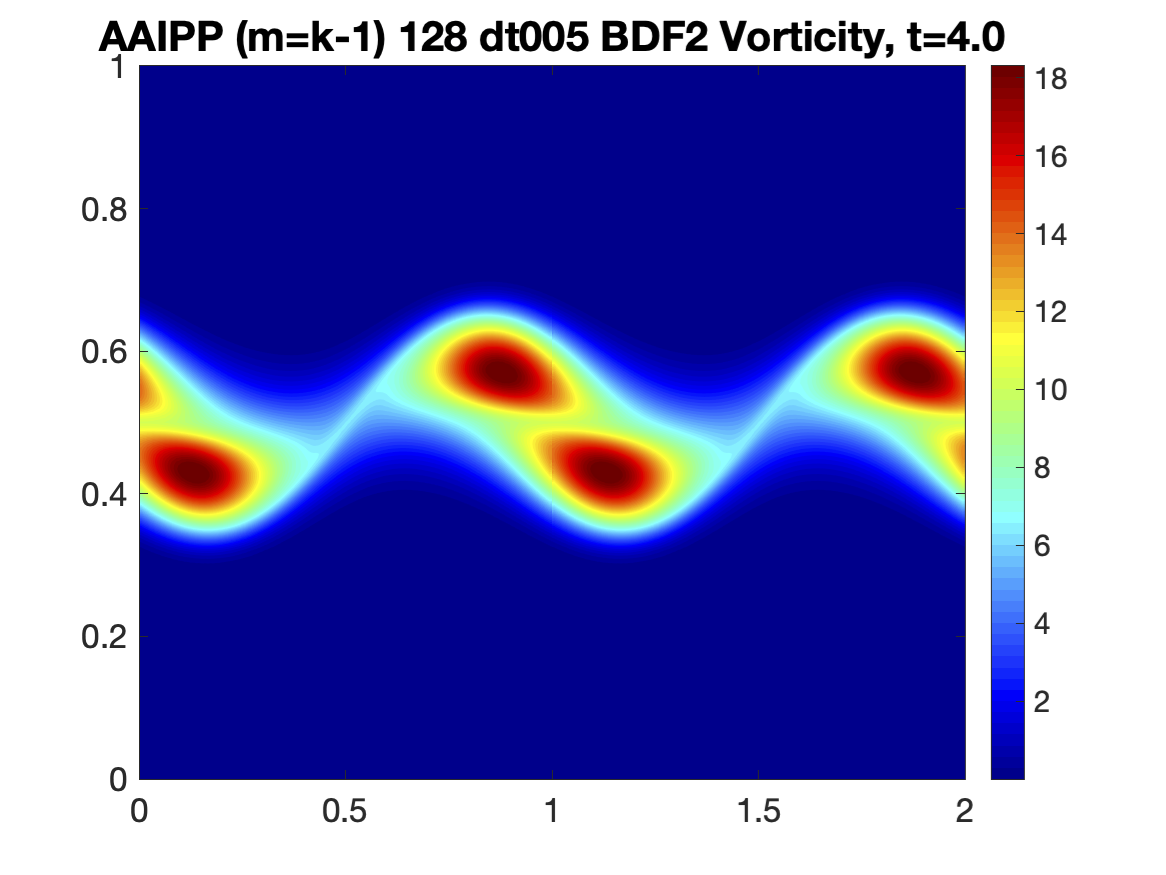} 
		\includegraphics[width=.3\textwidth, height=.23\textwidth,viewport=65 40 520 390, clip]{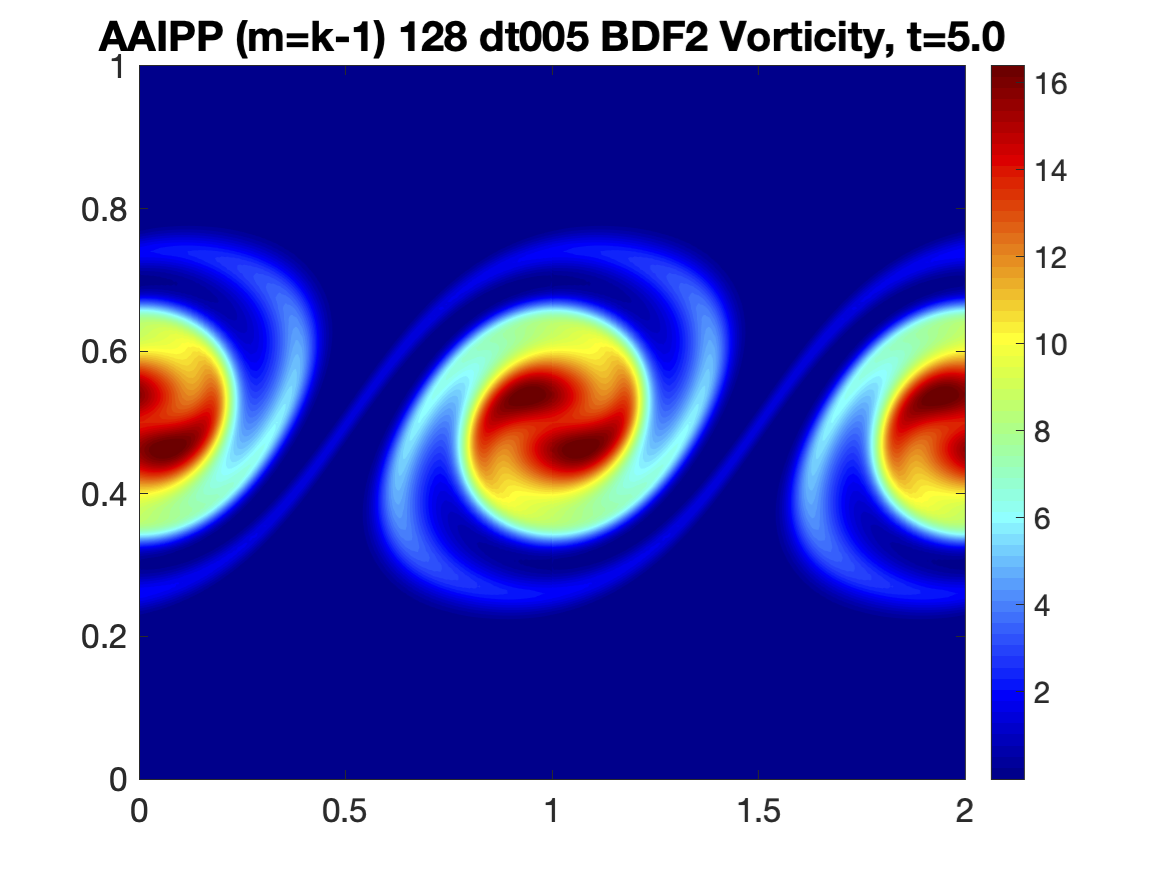}
	\end{center}
	\caption{\label{kh100}
		Shown above are $Re=100$ absolute vorticity contours for IPP solution, at times t=0, 1, 2, 3, 4, 5 and 6 (left to right, top to bottom).}
\end{figure}

\begin{figure}[!h]
	\begin{center}
		\includegraphics[width=.3\textwidth, height=.23\textwidth,viewport=0 0 520 400, clip]{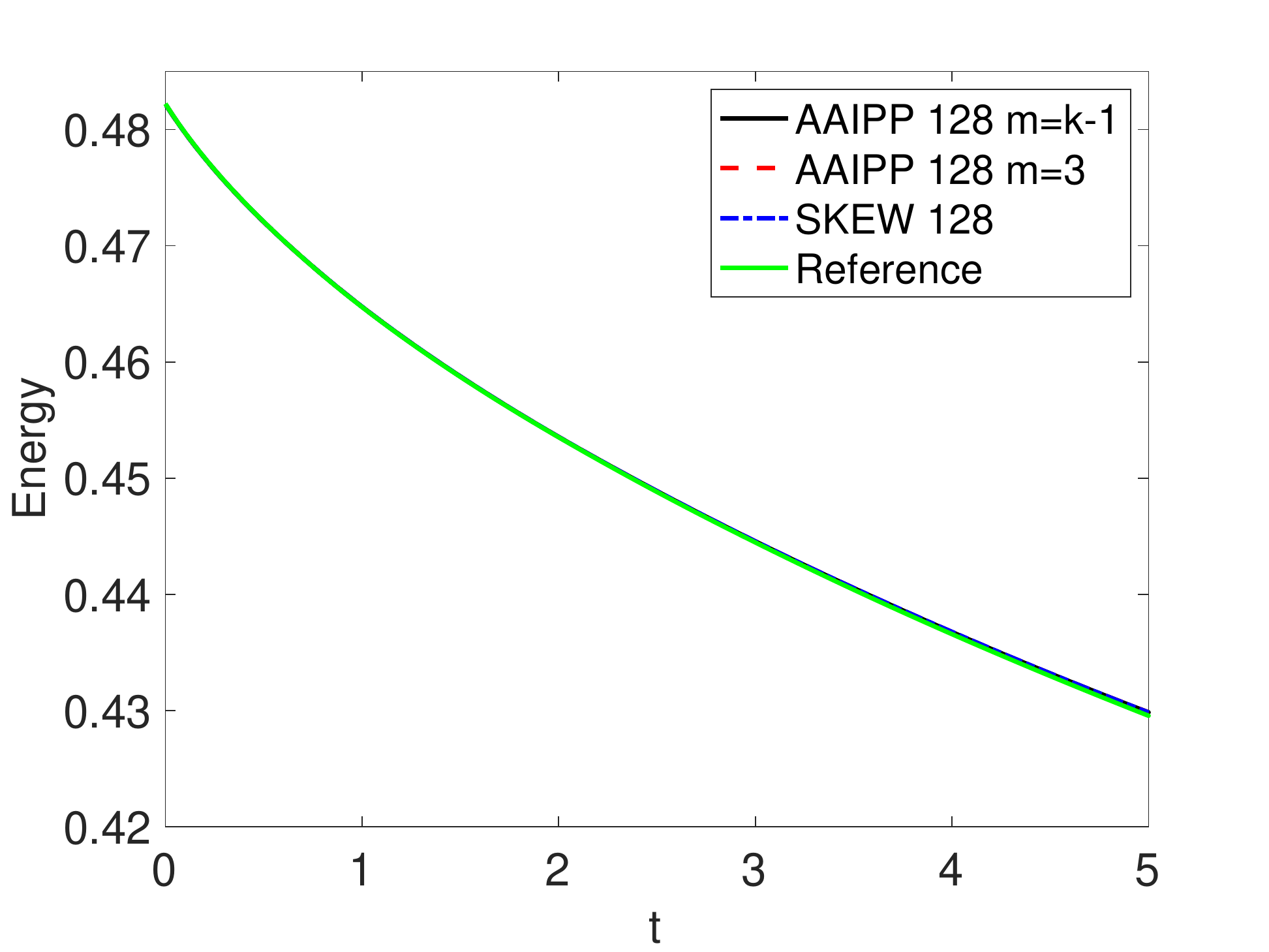}
		\includegraphics[width=.3\textwidth, height=.23\textwidth,viewport=0 0 520 400, clip]{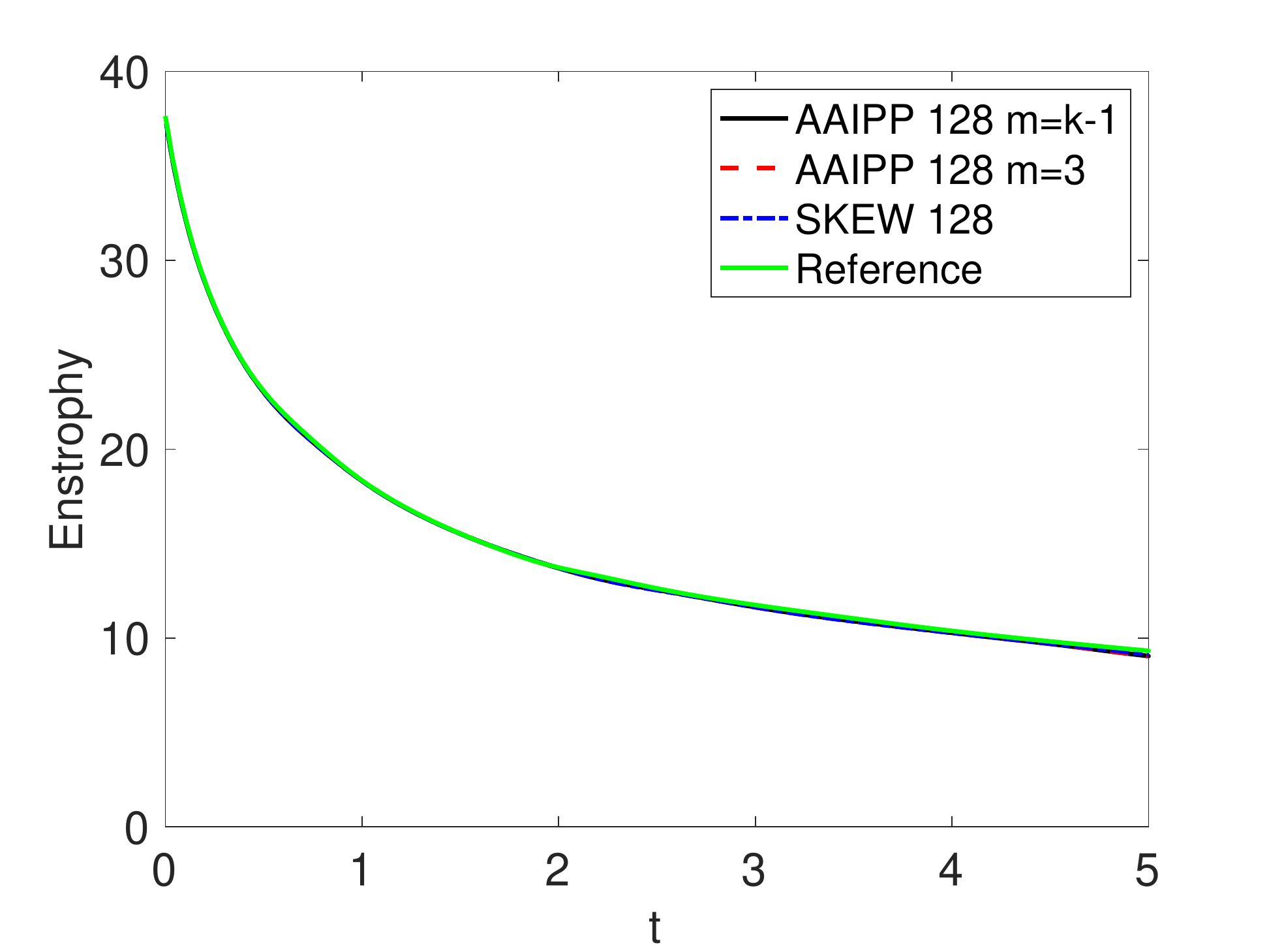}
		\includegraphics[width=.3\textwidth, height=.23\textwidth,viewport=0 0 520 400, clip]{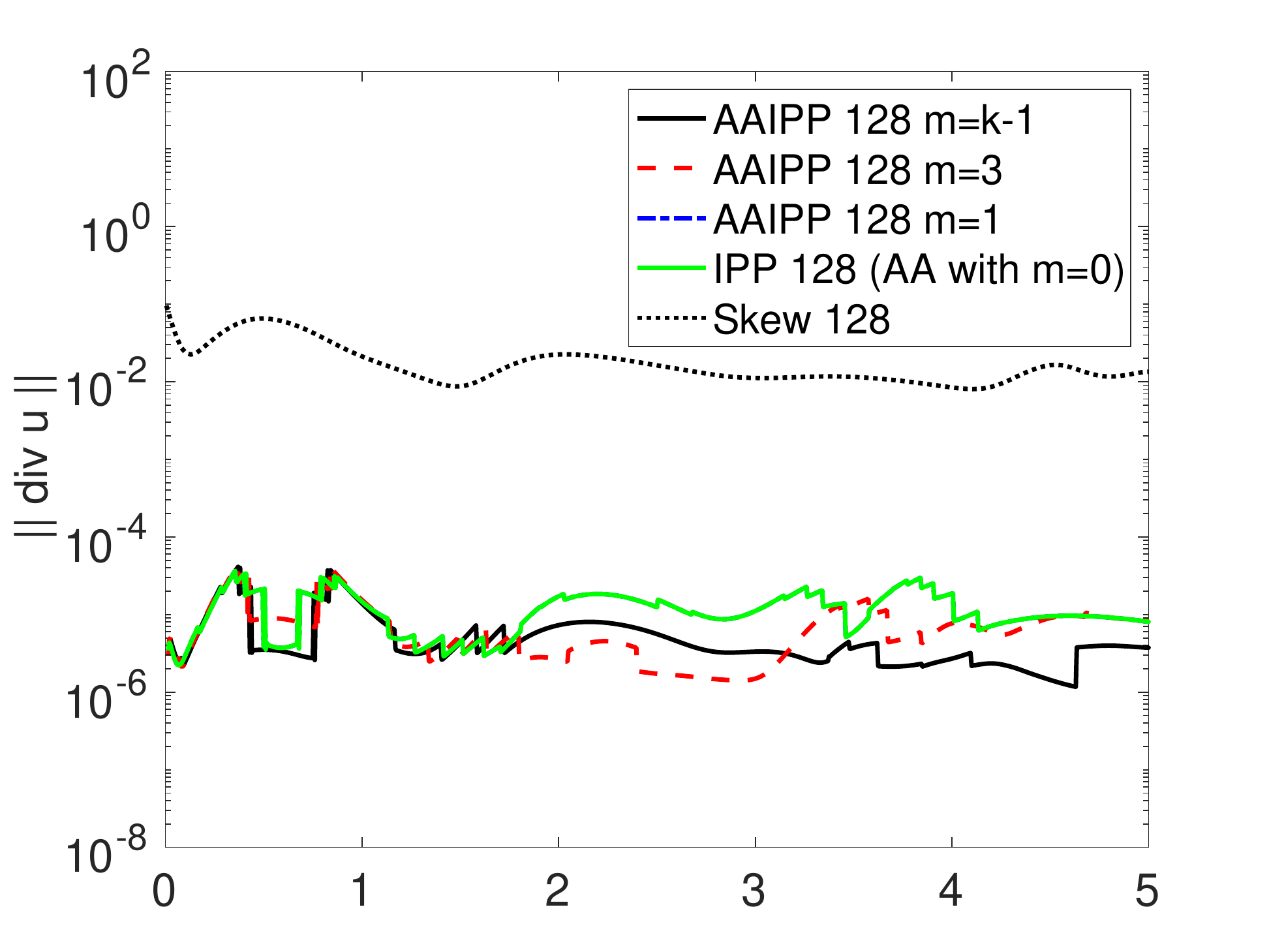}
	\end{center}
	\caption{\label{kh1002}
		Shown above are $Re=100$ energy, enstrophy and divergence error versus time for the IPP/AAIPP, SKEW and reference solutions from \cite{SJLLLS18}.}
\end{figure}

\begin{figure}[!h]
	\begin{center}
		\includegraphics[width=.5\textwidth, height=.3\textwidth,viewport=0 0 520 400, clip]{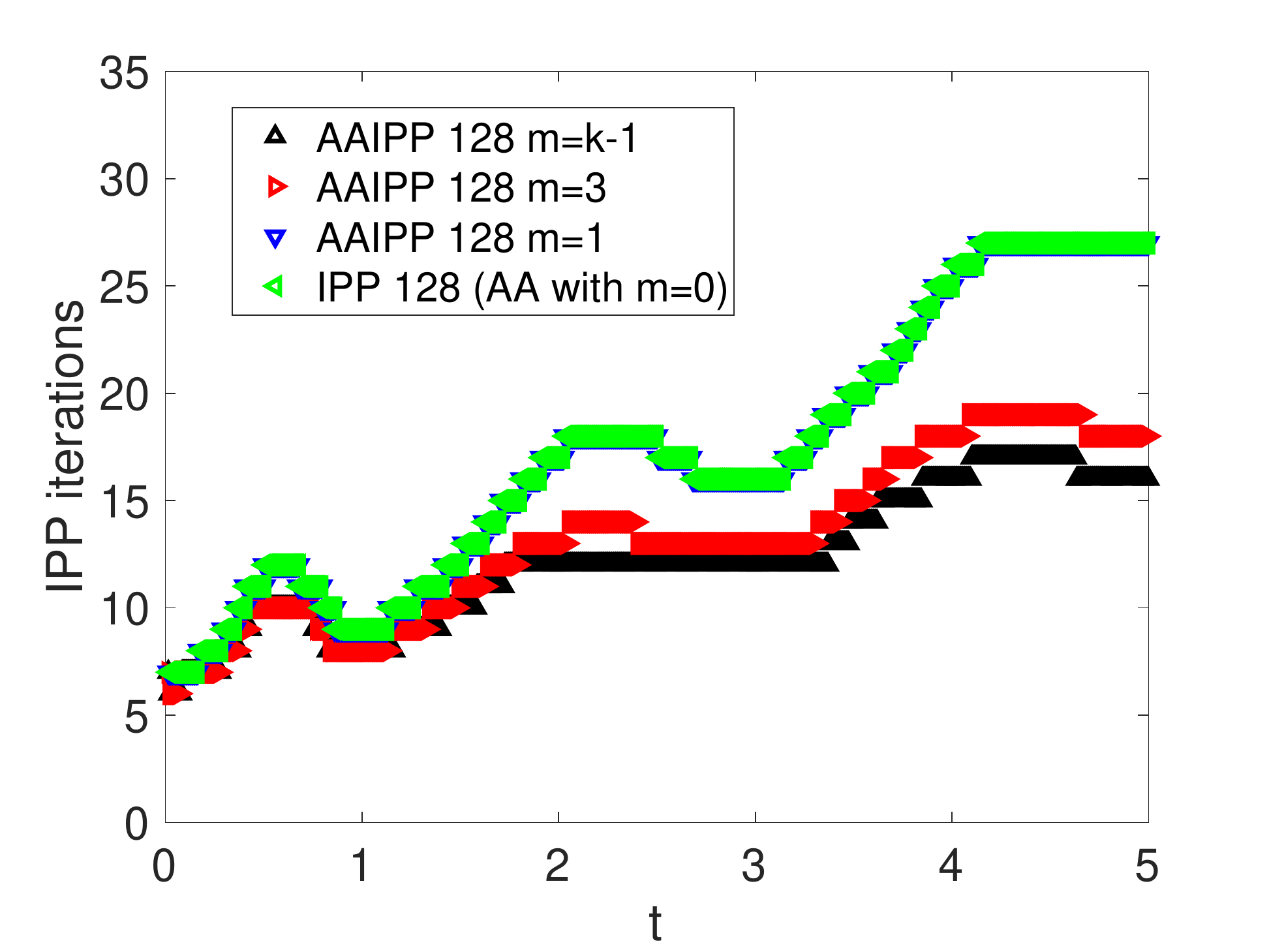}
	\end{center}
	\caption{\label{kh1003}
		Shown above are IPP/AAIPP total iteration counts at each time step, for $Re=100$ and varying $m$.}
\end{figure}

For our last test we consider a benchmark problem from \cite{SJLLLS18} for 2D Kelvin-Helmholtz instability.  This test is time dependent, and we apply the IPP/AAIPP method at each time step to solve the nonlinear problem.  The domain is the unit square, with periodic
boundary conditions at $x=0,1$.  At $y=0,1$, the no penetration boundary condition $u \cdot n=0$ is strongly enforced, along with a natural weak enforcement of the free-slip condition $(-\nu\nabla u \cdot n)\times n=0$.  The initial condition is given by
\[
u_0(x,y) = \left( \begin{array}{c} \tanh\left( 28 (2y-1) \right) \\ 0 \end{array} \right) + 10^{-3} \left( \begin{array}{c} \partial_y \psi(x,y) \\ -\partial_x \psi(x,y) \end{array} \right),
\]
where $\frac{1}{28}$ is the initial vorticity thickness, $10^{-3}$ is a noise/scaling factor, and 
\[
\psi(x,y) =  \exp \left( -28^2 (y-0.5)^2 \right) \left( \cos(8\pi x) + \cos(20\pi x) \right).
\]
The Reynolds number is defined by $Re=\frac{1}{28 \nu}$, and $\nu$ is defined by selecting $Re$.  Solutions are computed for both $Re=100$ and $Re=1000$, up to end time $T=5$.

Define $X=\{ v\in H^1(\Omega),\ v(0,y)=v(1,y),\ v\cdot n=0 \mbox{ at } y=0,1 \}$, and take $V = \{ v \in X,\ \| \nabla \cdot v \|=0\}$ and $X_h=P_2(\tau_h)\cap X$.  The problem now becomes at each time step: Find $u^{n+1}\in X_h \cap V$ satisfying
\begin{multline*}
\frac{1}{2\Delta t}(3u^{n+1},v) + (u^{n+1} \cdot\nabla u^{n+1},v) + \nu(\nabla u^{n+1},\nabla v) \\= (f,v) + \frac{1}{2\Delta t}(4u^n - u^{n-1},v)\  \forall v\in X_h \cap V.
\end{multline*}
The IPP iteration to find each $u^{n+1}$ is thus analogous to what is used for solving the steady NSE above including pressure recovery, but now with the time derivative terms and using the previous time step solution as the initial guess. 

For accuracy comparison, we also give results using the standard (nonlinear) BDF2 mixed formulation using skew-symmetry, which we will refer to as the SKEW formulation:
Find $u^{n+1}\in X_h$ and $p_h^{n+1} \in Q_h = P_1(\tau_h)\cap L^2_0(\Omega)$ satisfying
\begin{align*}
\frac{1}{2\Delta t}(u^{n+1},v) + b^*(u^{n+1},u^{n+1},v)+(&p^{n+1},\nabla \cdot v) + \nu(\nabla u^{n+1},\nabla v)  \\&= (f,v) + \frac{1}{2\Delta t}(u^n - 2u^{n-1},v), \\
(\nabla \cdot u^{n+1},q) &= 0,
\end{align*}
for all $(v,q)\in (X_h,Q_h)$. The nonlinear problem for SKEW is resolved using Newton's method, and since Taylor-Hood elements are being used, a large divergence error is expected.

For $Re=100$, a $h=\frac{1}{128}$ uniform triangular mesh was used, together with a time step of $\Delta t=0.005$.  The tolerance for the nonlinear solver was to reduce the $H^1$ relative residual to $10^{-6}$.   Simulations were performed with IPP, AAIPP with $m=1,3,k-1$, and SKEW.  The evolution of the flow can be seen in figure \ref{kh100} as absolute vorticity contours from the AAIPP solutions (all IPP and AAIPP solutions were visually indistinguishable), and these match those of the high resolution solution from \cite{SJLLLS18} and solutions from \cite{OR20}.   In addition to their plots of vorticity contours being the same, the IPP and AAIPP solutions yielded the same energy and enstrophy to five significant digits (i.e. they all give the same solution, as expected).  Figure \ref{kh1002} shows the energy, enstrophy and divergence of the IPP/AAIPP/SKEW solutions versus time, along with energy and enstrophy of the high resolution solutions from \cite{SJLLLS18}.  We observe that the energy and enstrophy solutions of IPP, AAIPP, and SKEW all match the high resolution reference solutions very well.  As expected, the IPP/AAIPP solutions have divergence error around $10^{-5}$, which is consistent with a relative residual $L^2$ stopping criteria of $10^{-8}$.  SKEW, however, has a large divergence error that is $O(10^{-2})$ despite a rather fine mesh and essentially resolving the flow; since Taylor-Hood elements are used, this large divergence error is not surprising \cite{JLMNR17}.

Figure \ref{kh1003} shows the number of iterations needed to converge IPP/AAIPP at each time step.  We observe that using AAIPP with $m=1$ offers no real improvement over IPP ($m=0$) in converging the iteration, however both $m=3$ and $m=k-1$ both offer significant improvement.  While at early iterations the larger $m$ choices give modest improvement, by $t=4$ the larger $m$ choices cut the iteration count from 27 to 16 at each time step.

\begin{figure}[!h]
	\begin{center}
		\includegraphics[width=.3\textwidth, height=.23\textwidth,viewport=65 40 520 390, clip]{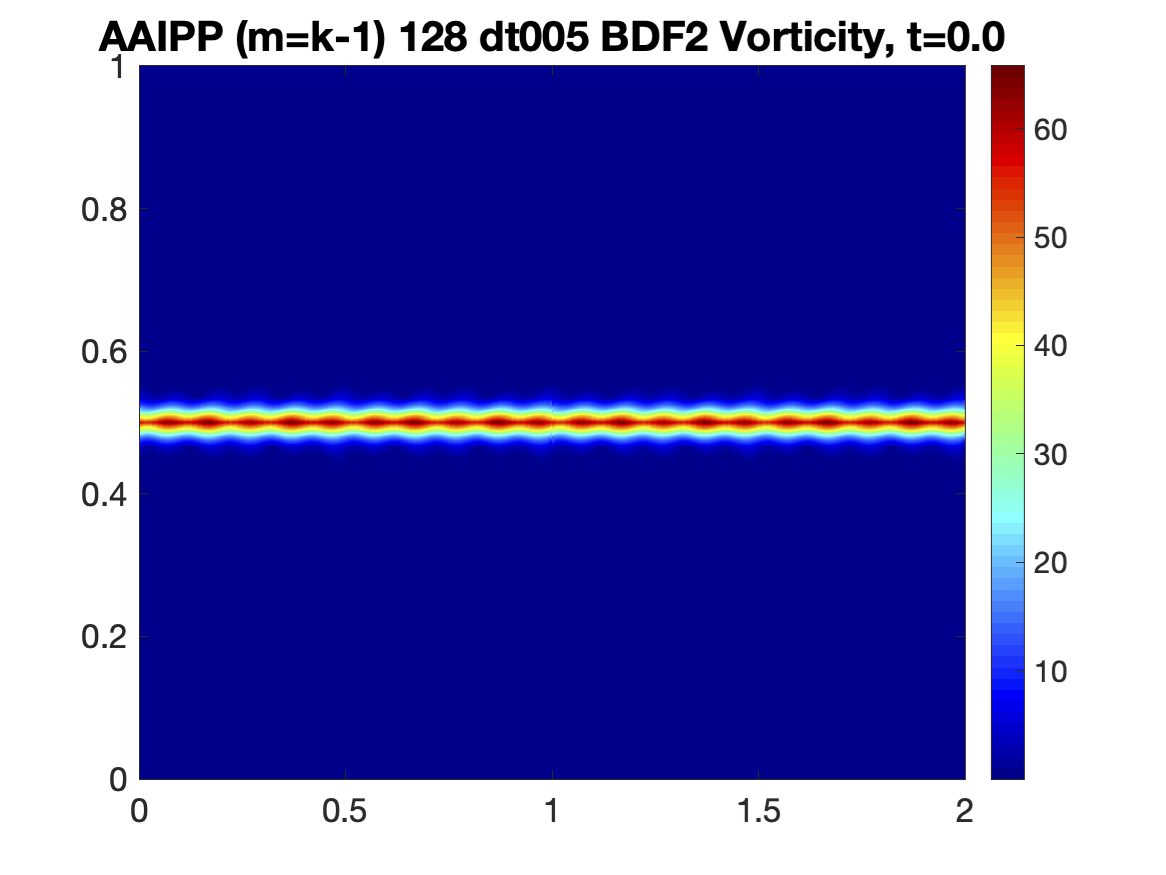}
		\includegraphics[width=.3\textwidth, height=.23\textwidth,viewport=65 40 520 390, clip]{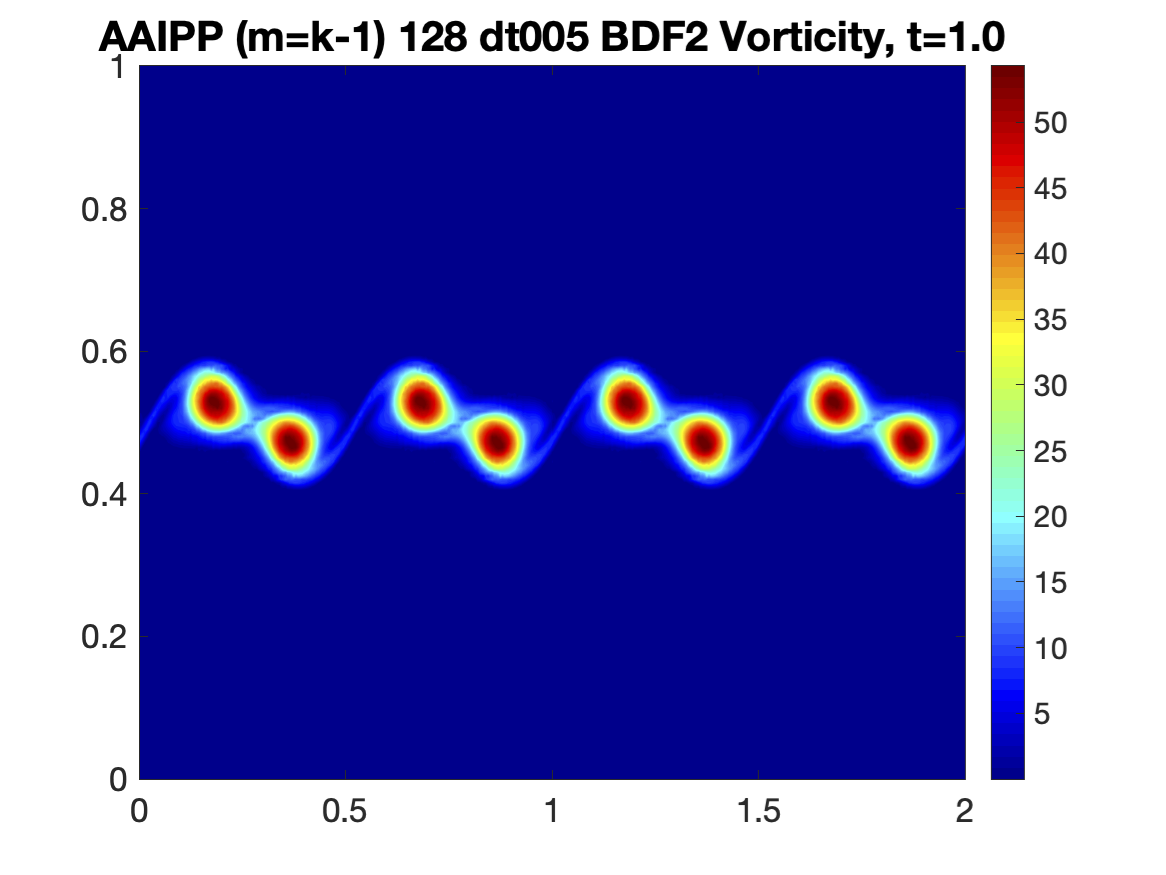} 
		\includegraphics[width=.3\textwidth, height=.23\textwidth,viewport=65 40 520 390, clip]{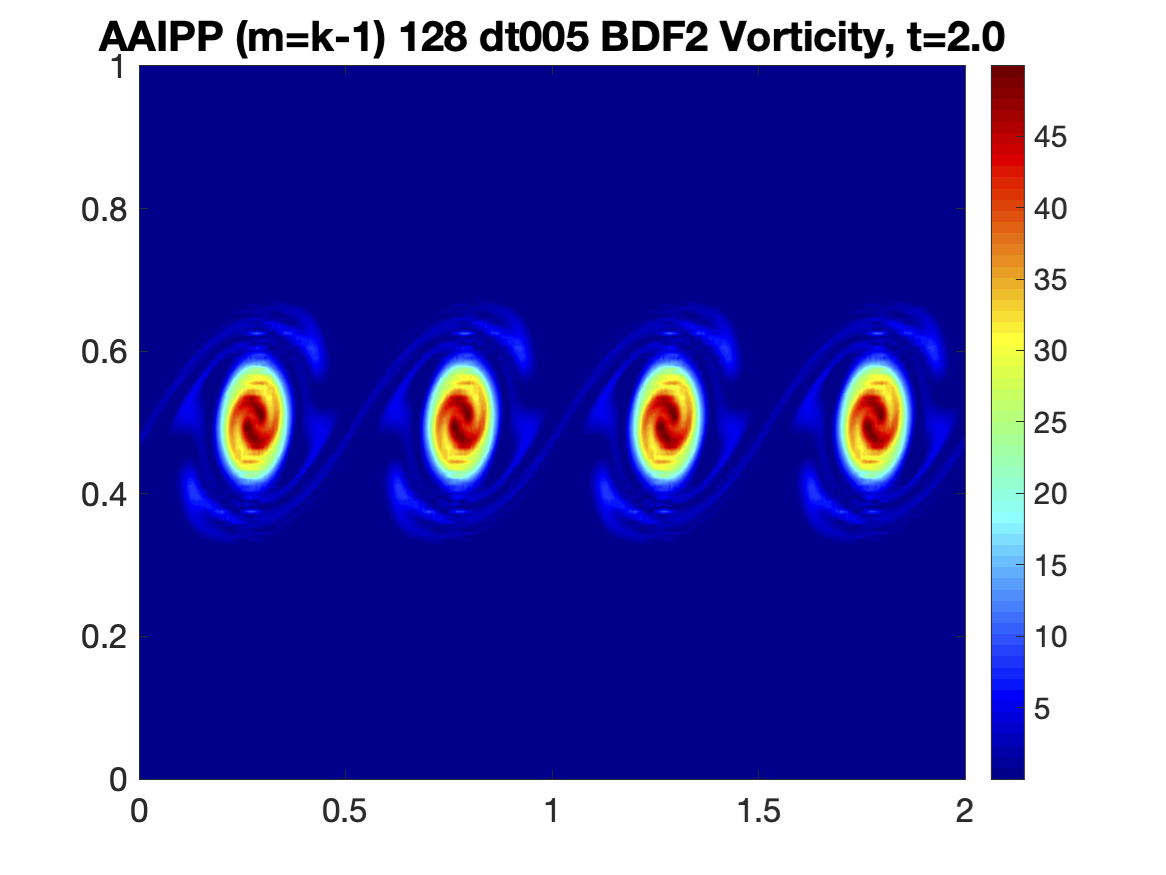}\\
		\includegraphics[width=.3\textwidth, height=.23\textwidth,viewport=65 40 520 390, clip]{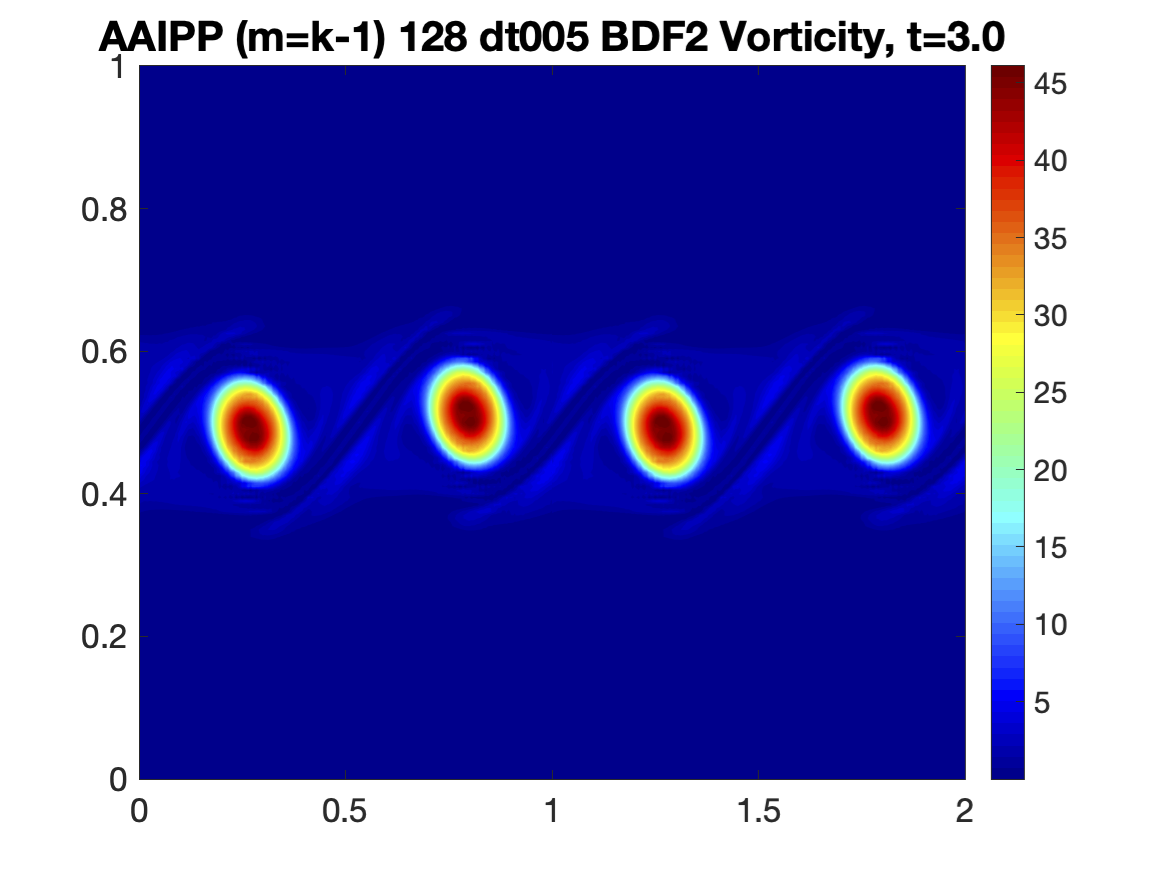}
		\includegraphics[width=.3\textwidth, height=.23\textwidth,viewport=65 40 520 390, clip]{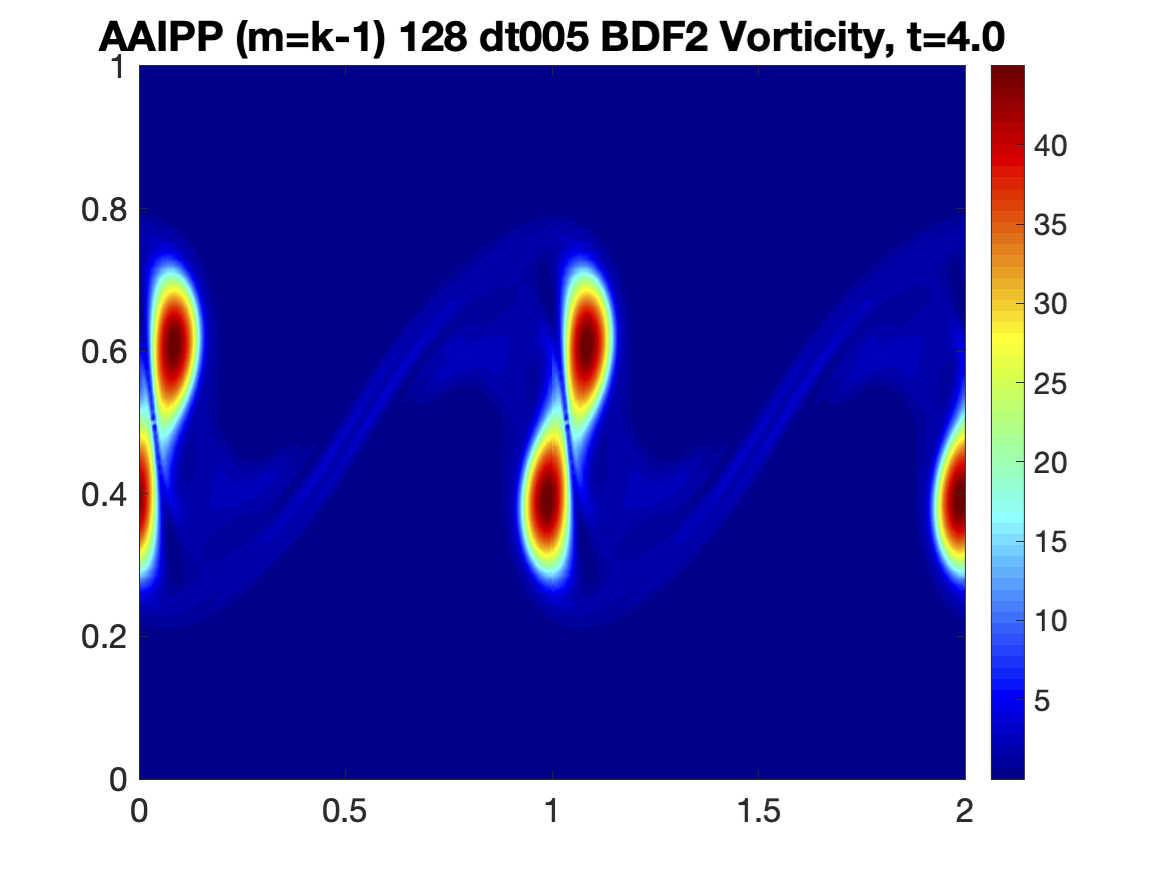} 
		\includegraphics[width=.3\textwidth, height=.23\textwidth,viewport=65 40 520 390, clip]{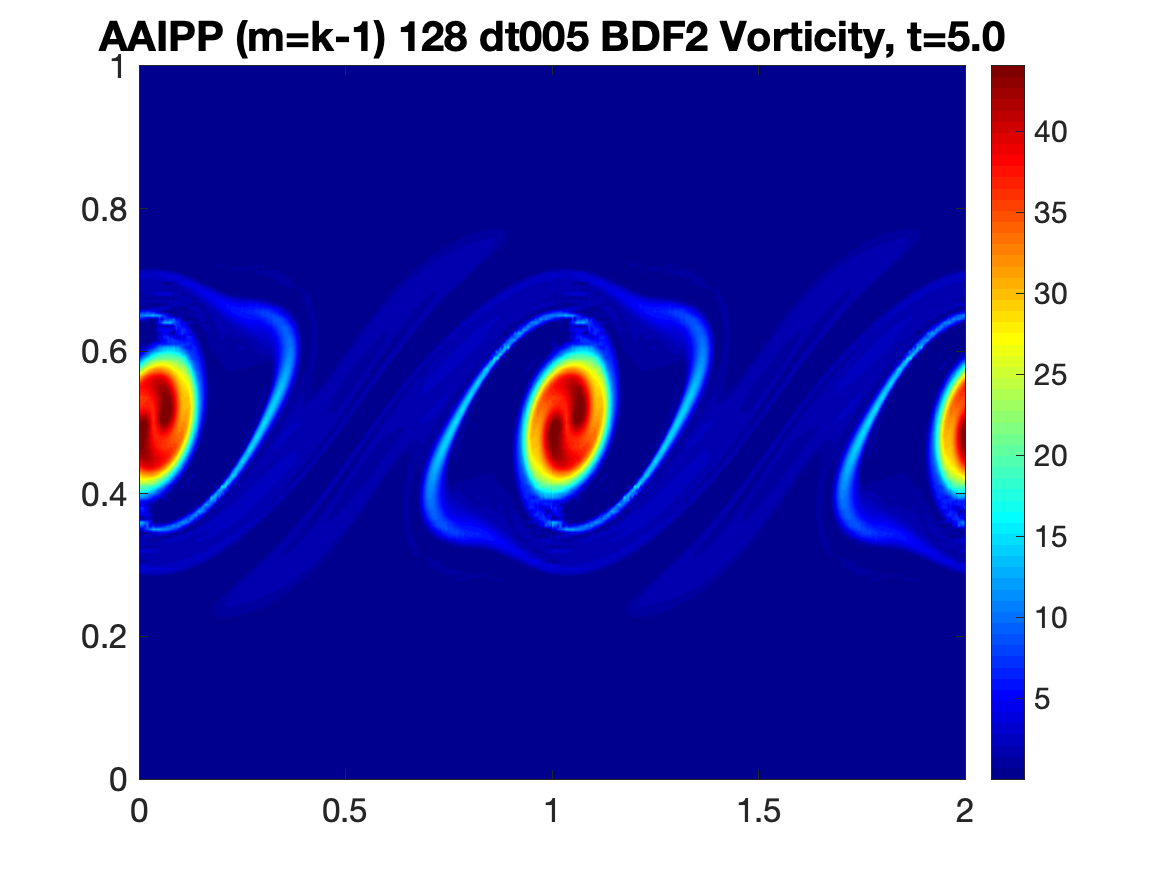}
	\end{center}
	\caption{\label{kh1000}
		Shown above are $Re=1000$ absolute vorticity contours for the iterated penalty solution, at times t=0, 1, 2, 3, 4, 5 and 6 (left to right, top to bottom).}
\end{figure}

\begin{figure}[!h]
	\begin{center}
		\includegraphics[width=.3\textwidth, height=.23\textwidth,viewport=0 0 520 400, clip]{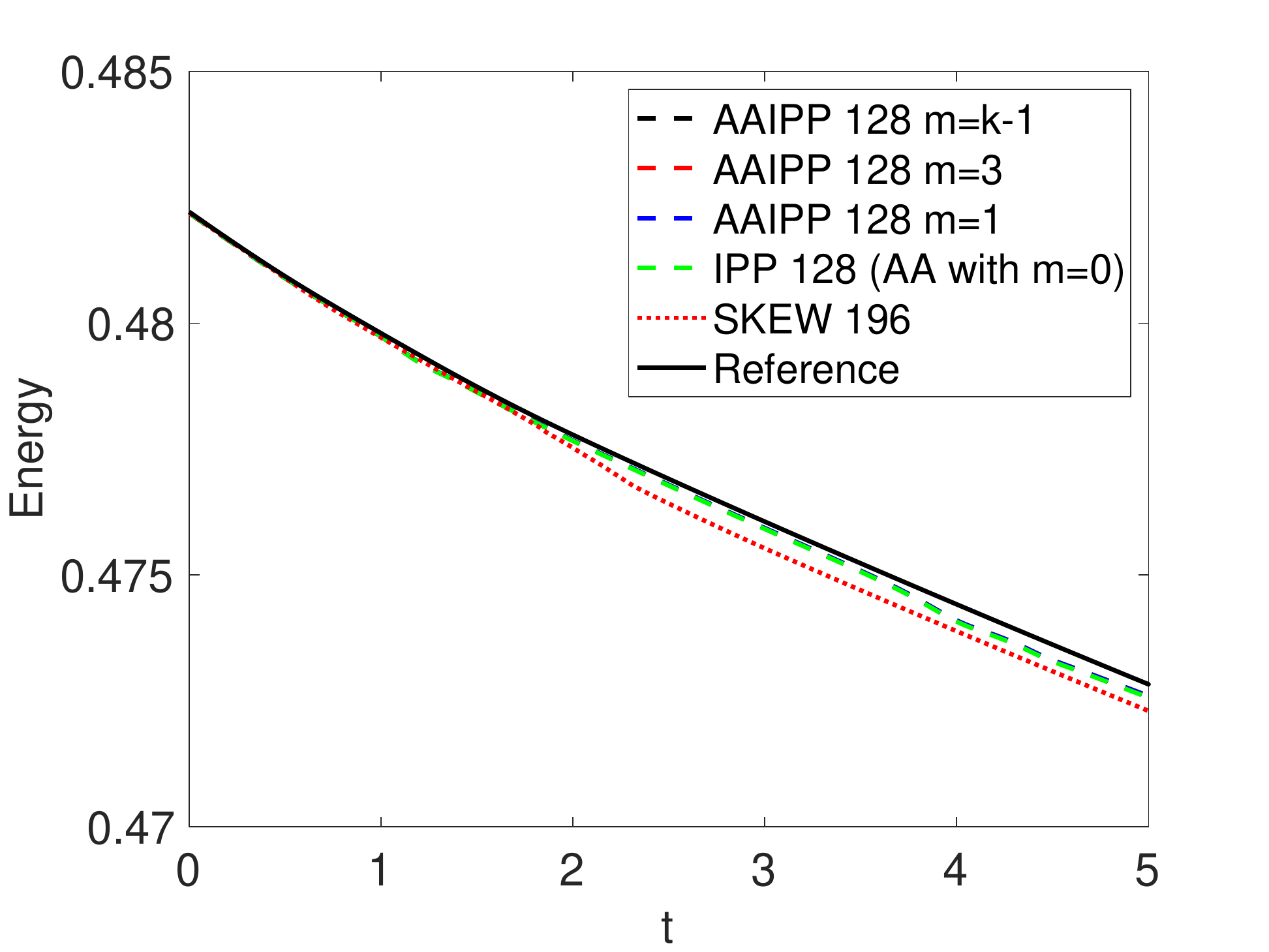}
		\includegraphics[width=.3\textwidth, height=.23\textwidth,viewport=0 0 520 400, clip]{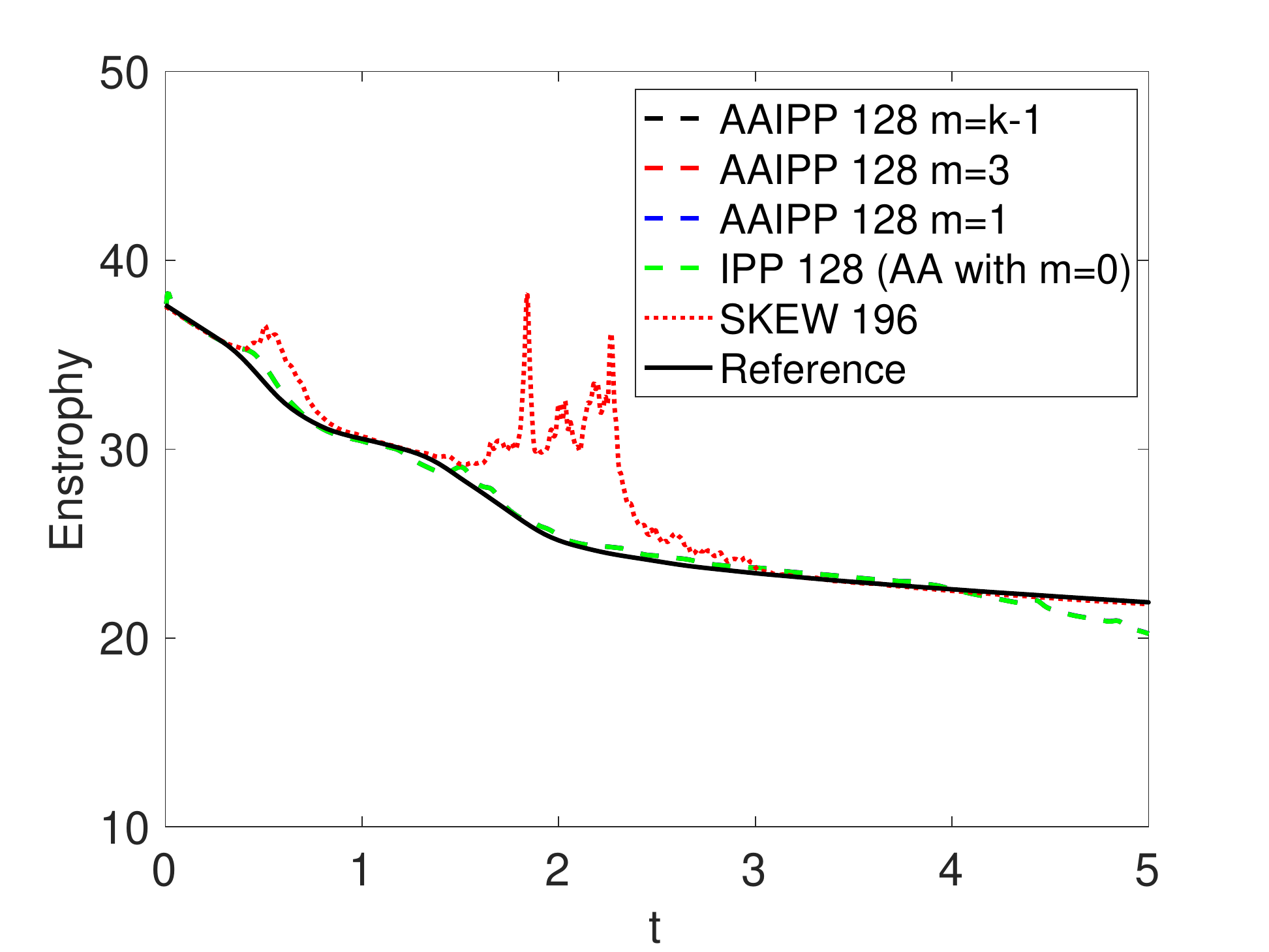}
		\includegraphics[width=.3\textwidth, height=.23\textwidth,viewport=0 0 520 400, clip]{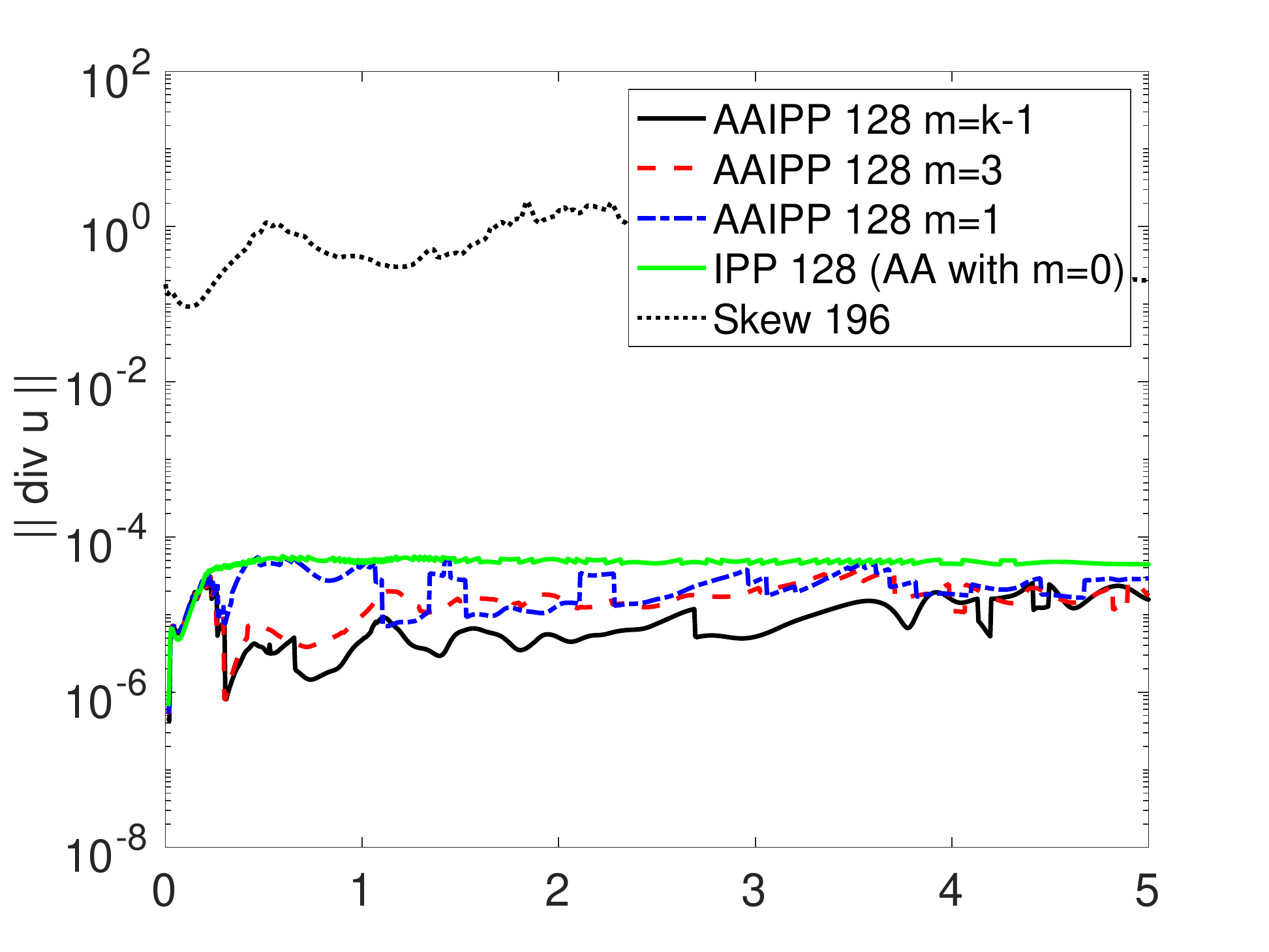}
	\end{center}
	\caption{\label{kh10002}
		Shown above are $Re=1000$ energy, enstrophy and divergence error versus time for the IPP/AAIPP, SKEW and reference solutions from \cite{SJLLLS18}.}
\end{figure}

\begin{figure}[!h]
	\begin{center}
		\includegraphics[width=.5\textwidth, height=.3\textwidth,viewport=0 0 520 400, clip]{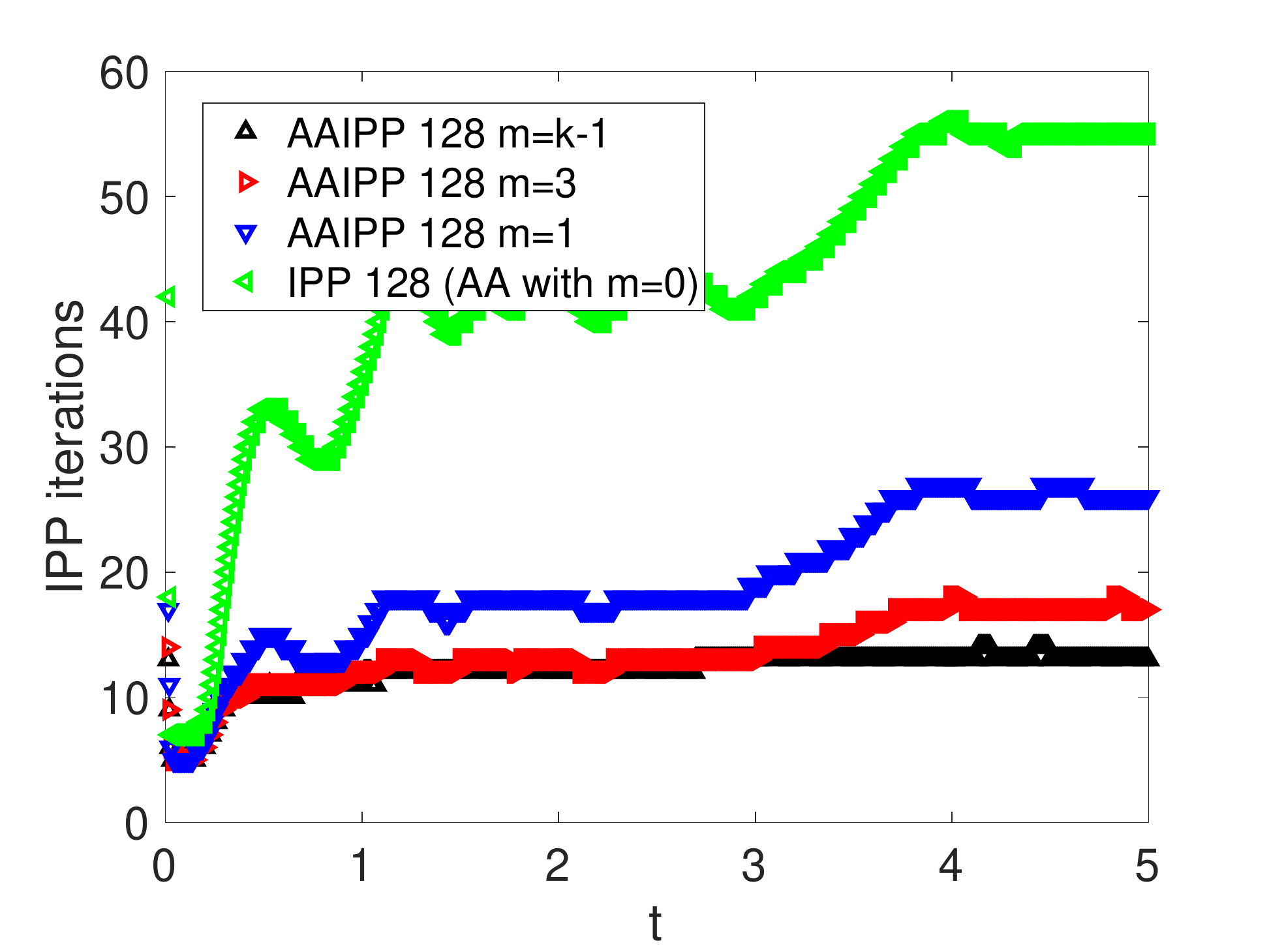}
	\end{center}
	\caption{\label{kh10003}
		Shown above are IPP/AAIPP total iteration counts at each time step, for $Re=1000$ and varying $m$.}
\end{figure}

For $Re=1000$, the IPP/AAIPP tests again used $h=\frac{1}{128}$, using the same setup as for the $Re=100$ case and again solutions are compared with a reference solution from \cite{SJLLLS18} and the SKEW solution (but here the SKEW solution uses a finer mesh with $h=\frac{1}{196}$.  As discussed in \cite{OR20}, even the $\frac{1}{196}$ mesh is not fully resolved for this Reynolds number.  The time step size was chosen to be $\Delta t=0.001$ for the IPP/AAIPP and SKEW simulations.  IPP/AAIPP vorticity contours are plotted in figure \ref{kh1000}, and match those from the reference solution qualitatively well (as discussed in \cite{SJLLLS18}, the evolution of this flow in time is very sensitive and it is not clear what is the correct behavior in time, even though it is clear how the flow develops spatially and how the eddies combine).  Figure \ref{kh10002} shows the energy, enstrophy and divergence of the computed and reference solutions (the reference solution has divergence on the order of roundoff error, and it is not shown), and we observe that the 1/196 SKEW solution gives the worst predictions of energy, enstrophy and (not surprisingly) divergence, even though IPP/AAIPP uses a significantly coarser mesh.  Finally, the impact of AA on the IPP iteration is shown in figure \ref{kh10003}, where we observe a significant reduction in iterations at each time step, with larger $m$ cutting the total number of iterations by a factor of four.  Hence overall, the AAIPP iteration is effective and efficient, and produces accurate divergence-free solutions.

\section{Conclusions}
In this paper, we studied IPP with penalty parameter $\epsilon=1$, and showed that while alone it is not an effective solver for the NSE, when used with AA and large $m$ it becomes very effective.  We proved the IPP fixed point function satisfies regularity properties which allow the AA theory of \cite{PR21} to be applied, which shows that AA applied to IPP will scale the linear convergence rate by the ratio gain of the underlying AA optimization problem.  We also showed results of three test problems which revealed AAIPP with $\epsilon=1$ is a very effective solver, without any continuation method or pseudo time-stepping.  While the classical IPP method is not commonly used for large scale NSE problems due to difficulties with linear solvers when $\epsilon$ is small, our results herein suggest it may deserve a second look since using AA allows for the penalty parameter $\epsilon=1$ to be used, which in turn will allow for effective preconditioned iterative linears solvers to be used such as those in \cite{HR13,BL12,OT14}.

%
%
%

\bibliographystyle{plain}
\bibliography{graddiv}

\end{document}